\documentclass[11pt]{article}
\usepackage{t1enc}
\usepackage{amsmath,amssymb,color}
\usepackage{graphicx}
\usepackage{enumitem,hyperref}
\usepackage{comment}
\newtheorem{Theorem}{Theorem}

\newtheorem{Proposition}{Proposition}
\newtheorem{Assumption}{Assumption}
\newtheorem{Lemma}{Lemma}
\newtheorem{Corollary}{Corollary}
\newtheorem{Remark}{Remark}
\newtheorem{Example}{Example}

\newenvironment{proof}[1][Proof]{\textbf{#1.} }{\hfill $\square$}

\newcommand{\blue}[1]{{\color{blue}#1}}
\newcommand{\red}[1]{{\color{red}#1}}

\def \eps{\varepsilon}

\def \bE{\mathbb{E}}

\def \cF{\mathcal{F}}

\def \bP{\mathbb{P}}
\def \1{\mathbf{1}}

\newcommand{\etamax}{\eta^\star}
\newcommand{\etamin}{\eta_\star}
\newcommand{\lambdamax}{\lambda^\star}

\newcommand*{\wt}{\widetilde}
\newcommand*{\ol}{\overline}

\begin{document}

\title{Time discretization of BSDEs with singular terminal condition using asymptotic expansion.}
\author{Julia Ackermann$^1$, Thomas Kruse$^2$, Alexandre Popier$^3$\\
\vspace{0.1cm} 
\\
\small{$^1$ Department of Mathematics \& Informatics, University of Wuppertal,}
    \vspace{-0.1cm}\\
    \small{Germany; e-mail: \texttt{jackermann}\textcircled{\texttt{a}}\texttt{uni-wuppertal.de}} \smallskip\\
    \small{$^2$ Department of Mathematics \& Informatics, University of Wuppertal,}
    \vspace{-0.1cm}\\
    \small{Germany; e-mail: \texttt{tkruse}\textcircled{\texttt{a}}\texttt{uni-wuppertal.de}} \smallskip\\
\small{$^3$ Laboratoire Manceau de Mathématiques, Le Mans Université,}
    \vspace{-0.1cm}\\
    \small{France; e-mail: \texttt{alexandre.popier}\textcircled{\texttt{a}}\texttt{univ-lemans.fr}} \smallskip\\
    }
\date{\today}

\maketitle

\begin{abstract}
    We consider a class of backward stochastic differential equations (BSDEs) with singular terminal condition and develop a numerical scheme to approximate their solution. 
    To this end, we extend an asymptotic development of the BSDE solution known from the power case, which arises from optimal liquidation problems, to more general generators. 
    This expansion allows to obtain a suitable approximation of the BSDE solution close to the terminal time. Using this as a terminal condition, we analyze the error of a backward Euler implicit scheme and detail its dependence on the terminal condition.
\end{abstract}

\tableofcontents

\section{Introduction}

In this paper, we consider backward stochastic differential equations (BSDEs in short) of the following form
\begin{equation} \label{eq:sing_BSDE}
- dY_t = \left( \frac{1}{\eta_t} f(Y_t) + \lambda_t \right) dt - Z_t dW_t, \quad t \in [0,T),
\end{equation}
with the singular terminal condition $\xi = +\infty$ a.s.\ and where $W$ is a $d$-dimensional Brownian motion 
on a filtered probability space $(\Omega,\cF_T,(\cF_t)_{t\in[0,T]},\bP)$ (which satisfies the usual conditions and that $(\cF_t)_{t\in[0,T]}$ is the filtration generated by $W$). 

This kind of BSDEs appears as the adjoint equation in optimal liquidation problems (see \cite{guea:16} for an overview) in finance. Indeed, in the paper \cite{anki:jean:krus:13}, the authors use the notion of a solution of a singular BSDE to provide a purely probabilistic solution of the stochastic control problem
to minimize 
\begin{equation}\label{eq:liquidationproblemcosts}
    \mathbb E\bigg[ \int_0^T \big( \zeta_t \lvert \dot{X}_t \rvert^p + \lambda_t \lvert X_t \rvert^p \big) \, dt \bigg]
\end{equation}
over all progressively measurable processes $X\colon \Omega \times [0,T] \to \mathbb R$ that satisfy $X_0=x \in \mathbb R$ and $X_T=0$ a.s.\ and have absolutely continuous paths. 
They show that 
the value function is given by $ |x|^p Y_0,$ where $Y$ is the solution of
\begin{equation} \label{eq:poly_bsde}
-dY_t  =-  \frac{p-1}{\zeta_t^{q-1}} |Y_t|^{q-1}Y_t dt + \lambda_t dt - Z_t dW_t,\quad t\in [0,T).
\end{equation}
Here $q>1$, and $p$ is its H\"older conjugate. 
Moreover, the optimal state process for the stochastic control problem also depends on $Y$ and is of the form 
$x \exp(-\int_0^t \zeta_u^{1-q} |Y_u|^{q-1} du)$
(see \cite{anki:jean:krus:13} for the details). 

\medskip

To obtain a solution for the BSDE \eqref{eq:sing_BSDE}, there are essentially two different approaches in the literature. 
\begin{itemize}
\item The truncation procedure developed in \cite{popi:06,anki:jean:krus:13,krus:popi:15}: the solution is obtained as the (increasing) limit of a sequence of solutions with bounded coefficients. 
\item An asymptotic development as in \cite{grae:hors:qiu:15,grae:hors:sere:18,grae:popi:21}: the expansion of the solution is written in terms of some known coefficients and the solution of a BSDE with a singular generator, but with a zero terminal condition.
\end{itemize}
The first procedure is easier to understand and the assumptions on the data $(f,\eta,\lambda)$ are much weaker in \cite{krus:popi:15} than in \cite{grae:popi:21}. Moreover, any terminal condition $\xi$ can be considered and not only $\xi = +\infty$. However, the rate of convergence for the approximating sequence is not known and thus cannot be used for numerics. The second approach requires stronger conditions on the parameters and $\xi = +\infty$, but gives a better understanding of the behavior of the solution close to the final time $T$.

Let us also remark that in the Markovian setting, BSDE \eqref{eq:sing_BSDE} is related to the reaction-diffusion partial differential equation (PDE in short)
\begin{equation} \label{eq:related_PDE}
\dfrac{\partial u}{\partial t} + \mathcal L u + \frac{1}{\eta(t,x)} f(u) +  \lambda(t,x) = 0,
\end{equation}
where $\mathcal L$ is the associated infinitesimal generator (here a second order differential operator) (see  \cite[Chapter 5.4]{pard:rasc:14} or \cite[Chapters 5 and 8]{zhan:17} for the link between BSDEs and PDEs). This relation is used in \cite{grae:hors:qiu:15,grae:hors:sere:18}, and such PDEs with the terminal condition $u(T,\cdot) = +\infty$ have been widely studied when $f(y) =- y|y|^{q-1}$ with $q > 1$ (see among others \cite{dynk:kuzn:98,marc:vero:99,popi:17}).  

\medskip

Hence, the theory for the BSDE \eqref{eq:sing_BSDE} is well established. But so far there is no work concerning the numerical computation of the solution. For practical applications (liquidation in finance), this lack is a main drawback. In \cite[Section 5]{anki:jean:krus:13} some particular examples are developed for which an explicit solution $Y$ can be computed for \eqref{eq:poly_bsde}; nevertheless, in general, an explicit expression is not available neither for the BSDE \eqref{eq:poly_bsde} nor for the BSDE \eqref{eq:sing_BSDE}. To the best of our knowledge, there is no work on numerically solving the previous PDE~\eqref{eq:related_PDE} by analytical methods. 

\medskip

The main goal of this paper is to give a numerical scheme for solving BSDEs of the type \eqref{eq:sing_BSDE}. There is a huge amount of literature about algorithms to solve BSDEs when the terminal value is integrable, which can be roughly classified as follows:
\begin{itemize}
\item Backward Euler schemes (among many others \cite{gobe:lemo:wari:05,lion:dosr:szpr:15,zhan:04} and \cite[Chapter 5]{zhan:17}). It requires that the terminal value $\xi$ is integrable, hence it cannot be used when $\xi = +\infty$. 
\item Picard iterations (\cite{hutz:jent:krus:23,hutz:krus:22}). Note that the solution of the BSDE with singular terminal condition is not obtained by a fixed-point argument. As previously mentioned, the truncation approximation is based on monotone convergence without any rate of convergence. 
\item Forward schemes via branching processes  (\cite{henr:tan:touz:19,henr:tan:touz:14}). These are adapted for power-type generators; but with bounded terminal condition for the BSDE. 
\end{itemize}
None of them can be directly used in our framework. We proceed in two steps. 
\begin{enumerate}
\item We find a good approximation of the solution of the BSDE on a neighborhood of $T$, that is we find $\Delta > 0$, $\alpha > 0$ and an explicit, adapted, bounded stochastic process $\xi$ such that the $L^\infty$-norm of $Y- \xi$ is controlled by $t \mapsto C(T-t)^\alpha$ on $[T-\Delta,T]$. See Assumption \ref{ass:approximation} below. 
\item We can use the backward Euler implicit scheme starting at the time $T-\Delta$ from the position $ \xi_{T-\Delta}$. 
\end{enumerate}
The key idea for the first step is to extend the asymptotic development of the solution as proposed in  \cite{grae:hors:qiu:15,grae:hors:sere:18,grae:popi:21} to more general generators $f$; up to now, the asymptotic development was known only for the power case coming from the optimal liquidation problem. 

Then we derive a global error of our approximation technique of $Y$. The discretization error of the backward Euler scheme has been extensively studied, including in settings more general than ours (see, e.g., \cite{gobe:lemo:wari:05,gobe:lemo:wari:06,lion:dosr:szpr:15}). Nevertheless, in these studies, many constants appear that actually depend on the terminal condition, either through its value or its regularity (e.g., its Lipschitz constant). When the terminal condition is fixed, this dependency is of little importance, since the error typically takes the form 
$C h^\beta$, where 
$h$ is the discretization step size and $C$ depends on all the parameters. The explicit dependency is buried within the error analysis and is difficult to assess (see for example \cite[Theorems 5.2.4 and 5.3.3]{zhan:17}). However, in our case, the norm of the terminal condition increases 
when the parameter $\Delta$ gets small.
Therefore, it is crucial to consistently keep track of this dependency in all the constants that appear in the discretization error. 
The numerical solution of this type of BSDE enables the computation of the corresponding PDE \eqref{eq:related_PDE} solution with an infinite terminal condition, through the relation 
$u(t,x)=Y^{t,x}_t$. This approach therefore provides an effective algorithmic framework for solving such PDEs as well. To the best of our knowledge, no such algorithm currently exists within the framework of PDE theory.

\paragraph{Breakdown of the paper.}

Following a review of existing results on BSDEs with singular terminal conditions, Section \ref{sect:time_discre} of the paper is devoted to analyzing the error of the implicit Euler scheme. As previously mentioned, we present a self-contained analysis in order to accurately track the dependence on all relevant parameters (see Theorem \ref{thm:error}). To the best of our knowledge, this is the first instance in which the discretization error is quantified for a BSDE with a singular terminal condition. In the subsequent Section \ref{sect:approx}, we establish sufficient conditions on the generator $f$ under which the approximation required for the numerical implementation is valid (see Theorem \ref{thm:gene_case}), thereby extending known results for power-type nonlinearities. 
Technical proofs are relegated to the appendix (Section \ref{sect:appendix}).

\paragraph{Notations.}
Evoke that for $p\geq 1$
\begin{itemize}
\item $\mathbb D^p(0,T)$ is the space of all adapted c\`adl\`ag\footnote{French acronym for right-continuous with left limits} processes $X$ such that
$$\mathbb E \bigg[  \sup_{t\in [0,T]} |X_t|^p \bigg] < +\infty.$$
\item $\mathbb H^p(0,T)$ is the subspace of all predictable processes $X$ such that
$$\bE \bigg[ \bigg( \int_0^T |X_t|^2 dt\bigg)^{\frac{p}{2}} \ \bigg] < +\infty.$$
\item $\mathbb S^p(0,T) = \mathbb D^p(0,T) \times \mathbb H^p(0,T)$ and $\mathbb S^\infty(0,T) = \bigcap_{p\geq 1} \mathbb S^p(0,T)$. 
\end{itemize}
Whenever the notation $T-$ appears in the definition, we mean the set of all processes whose restrictions satisfy the respective property when $T-$ is replaced by any $T-\eps$, $\eps > 0$. For example, $\mathbb S^p(0,T-) = \bigcap_{\eps > 0}\mathbb S^p(0,T-\eps)$.
For any $t\in[0,T]$, conditional expectations with respect to $\cF_t$ are denoted by $\mathbb E_t[\cdot]$.

\subsection*{Setting and known results}

Here and in the rest of the paper, we keep the conditions of \cite{grae:popi:21}, that is, $T$ is fixed and we have:
\begin{Assumption} \label{ass:general_setting}
$\ $ 
\begin{enumerate}[label={\rm \textbf{(A\arabic*)}}]
\item \label{A1} {\it There exist three constants $0 < \etamin < \etamax$ and $\lambdamax \geq 0$ such that a.s. for any $t\in [0,T]$ it holds 
$$\etamin \leq \eta_t \leq \etamax,\qquad 0 \leq \lambda_t \leq \lambdamax.$$}
\item  \label{A2} {\it The process $\eta$ is an It\^o process
\begin{equation} \label{eq:eta-ito}
d\eta_t=b^\eta_t\,dt+\sigma^\eta_t\,dW_t, \quad t\in[0,T],
\end{equation}
such that the processes $b^\eta$ and $\sigma^\eta$ belong to $L^\infty(\Omega\times [0,T] ;\mathbb R)$ and $L^\infty(\Omega\times [0,T] ;\mathbb R^{1\times d})$, respectively. }
\item \label{A3} {\it The function $f:\mathbb R \to \mathbb R$ is continuous and non-increasing with $f(0)=0$. Furthermore, on $[0,\infty)$, $f$ is concave and of class $C^2$. }
\item \label{A4} {\it The function $G: (0,\infty) \to \mathbb R$ 
\begin{equation} \label{eq:def_G}
G(x):=\int_x^\infty \frac{1}{-f(y)}\,dy, \quad x \in (0,\infty),
\end{equation}
is well-defined on $(0,\infty)$ and with values in $(0,\infty)$.}
\end{enumerate}
\end{Assumption}

Furthermore, also as in \cite{grae:popi:21}, we denote by 
$$\phi\colon (0,\infty)\to(0,\infty)$$ 
the inverse of the function $G$ of \ref{A4}. Note that $\phi$ solves the ODE $\phi'=f\circ \phi$ with the initial condition $\lim_{x\to 0} \phi(x)=\infty$.

\begin{Example}
Starting from any solution $X$ of an SDE with Lipschitz coefficients $b$ and $\sigma$ such that $|b(t,0)| + \|\sigma(t,0)\| \leq C$, apply It\^o's formula to $\eta=\varphi(X)$, with $\varphi(x) = \dfrac{\etamax-\etamin}{\pi} \arctan(x) + \dfrac{\etamax+\etamin}{2}$. Then $\eta$ satisfies all the required conditions  {\rm \ref{A1}} and {\rm \ref{A2}}.
\end{Example}

From \cite[Proposition 4]{grae:popi:21}, under the previous Assumption \ref{ass:general_setting}, 
the BSDE \eqref{eq:sing_BSDE} has a unique non-negative solution $(Y,Z) \in \mathbb S^\infty(0,T-)$:
\begin{itemize}
\item A.s., for any $0\leq s \leq t < T$, 
$$Y_s = Y_t + \int_s^t \left[ \frac{1}{\eta_u} f(Y_u) + \lambda_u \right] du - \int_s^t Z_u dW_u.$$
\item $Y_t \geq 0$ for any $t \in [0,T)$, a.s.  
\item A.s. $\displaystyle \lim_{t\to T} Y_t =+\infty$.
\end{itemize}

\begin{Remark}[Comments on \ref{A1} and \ref{A2}]

In \cite{krus:popi:15}, it is only supposed that $1/\eta$ is in $L^1(\Omega\times [0,T] ;\mathbb R)$ and that $\eta^{p-1} \in L^\ell(\Omega\times [0,T] ;\mathbb R)$ for some $\ell > 1$, where $f(y) \le -y|y|^{q-1}$ holds for all $y \ge 0$ and $q,p \in (1,\infty)$ are H\"older conjugates. Moreover, a weak integrability condition is imposed on $\lambda$ and the non-negativeness of $\lambda$ could be relaxed in \cite{krus:popi:15}. Thus the two conditions \ref{A1} and \ref{A2} can be weakened for the existence of a minimal solution (see \cite[Theorem~1]{krus:popi:15}). 

Nonetheless, even if the terminal condition of the BSDE \eqref{eq:sing_BSDE} is bounded, the conditions of \cite{krus:popi:15} on $\eta$ and $\lambda$ have to be strengthened to apply known numerical schemes; see for example \cite[Condition (HY0)]{lion:dosr:szpr:15}, which essentially corresponds to {\rm \ref{A1}}. 
\end{Remark}

\begin{Remark}[Comments on \ref{A3} and \ref{A4}]
Again, existence can be obtained without the regularity condition on $f$ (see \cite[Theorem 1]{krus:popi:15}). The non-increasing condition can also be replaced by a more general monotonicity condition (see Assumption A1 and Remark 1 in [20]). The regularity is used to obtain the asymptotic expansion of $Y$, together with uniqueness of the solution. 

In \cite{krus:popi:15}, it is assumed that there exists some constant $q > 1$ such that for any $y \geq 0$ it holds $f(y)  \leq -y|y|^{q-1}$. In this case, {\rm \ref{A4}} is satisfied.  
\end{Remark}

Under the additional condition \ref{A5} of Section \ref{sect:approx}, it is proved in \cite[Theorem 10]{grae:popi:21} that $Y$ 
admits the representation 
\begin{equation} \label{eq:asymp_beha_Y_A}
Y_t = \phi \bigg(\bE_t \bigg[  \int_t^T \frac{1}{\eta_s} ds \bigg]\bigg) - \phi'\bigg( \bE_t \bigg[  \int_t^T \frac{1}{\eta_s} ds \bigg]\bigg) \widetilde H_t,\quad \forall t \in [0,T), 
\end{equation}
where $\phi$ is the inverse of $G$, and the non-negative process $\widetilde H$ is the unique solution of a BSDE with terminal condition 0 and with a singular\footnote{In the sense of the paper \cite{jean:reve:14}.} generator.  This expansion \eqref{eq:asymp_beha_Y_A} shows that there is a one-to-one correspondence between $Y$ and $\widetilde H$. However, it requires to compute the conditional expectation $\bE_t[  \int_t^T \frac{1}{\eta_s} ds ]$. Moreover $\widetilde H$ is obtained by a monotonicity argument without any rate of convergence, and due to the singularity of its related generator, the simulation of $\widetilde H$ is an open problem. Hence this result can not be used directly for a numerical application. 

In the special case where $f(y) = - y|y|^{q-1}$, $q>1$, in \cite[Section 4.1]{grae:popi:21} (see also \cite[Appendix]{caci:deni:popi:25} for more details), a different expansion is proved: with $p$ being the H\"older conjugate of $q$, it holds for all $t\in [0,T)$ that 
\begin{equation}\label{eq:asymp_exp_power}
Y_t = \phi \left( \dfrac{T-t}{\eta_t} \right) - \phi'\left( \dfrac{T-t}{\eta_t} \right)  H_t = \left(\dfrac{ (p-1) \eta_t}{T-t} \right)^{p-1} +  \left(\dfrac{ (p-1) \eta_t}{T-t} \right)^{p} H_t.
\end{equation}
The drawback is that $H$ is not non-negative. But the advantages are the disappearance of the conditional  expectation and that an upper bound on $|H|$ of the form $C(T-t)^2$ is available, which shows that  
\begin{equation} \label{eq:behavior_T}
Y_t - \xi_t := Y_t - \bigg(\dfrac{ (p-1) \eta_t}{T-t} \bigg)^{p-1} =  O\big((T-t)^{2-p}\big).
\end{equation}
Hence, for $1< p<2$, that is $q>2$, we have a good approximation of the solution $Y$ close to the time $T$.

In the rest of the paper, we extend the expansion \eqref{eq:asymp_exp_power} to more general generators and explain how to use the property \eqref{eq:behavior_T} for numerical purposes.

\section{Time discretization of singular BSDEs using the implicit Euler scheme} \label{sect:time_discre}

Let us evoke the following result (see \cite[Lemma 3]{grae:popi:21}): for all $0\leq t < T$, a.s.\ 
\begin{equation}\label{eq:boundofYintermsofTheta}
    0\leq Y_t \leq \Theta(T-t), 
\end{equation}
where $\Theta$ solves on $(0,\infty)$ the ODE 
\begin{equation}
    \Theta' = \dfrac{f(\Theta)}{\etamax} + \lambdamax
\end{equation}
with the initial condition $\lim_{x\to 0} \Theta(x) = \infty$. 
Hence for any $\Delta \in (0,T]$, the process $Y$ is bounded on $[0,T-\Delta]$.

In this section, we assume that: 
\begin{Assumption}[Approximation of $Y$] \label{ass:approximation}
There exist $\tau \in [0,T)$ and an explicit, adapted stochastic process $(\xi_t)_{t\in [\tau,T)}$, $C\in[1,\infty)$ and $\alpha \in (0,\infty)$ such that for all $t \in [\tau,T)$ we have
\begin{equation}\label{eq:cond_term}
\|Y_{t}- \xi_t\|_{\infty}\le C (T-t)^\alpha.
\end{equation}
\end{Assumption}

Under this assumption, together with 
Assumption~\ref{ass:general_setting}, 
we deduce using \eqref{eq:boundofYintermsofTheta} that 
\begin{equation} 
  \label{eq:a_priori_estim_term_cond}
\forall \Delta \in (0,T-\tau]\colon\quad \|\xi_{T-\Delta} \|_{\infty} \leq \Theta(\Delta) + C\Delta^\alpha .  
\end{equation}

\begin{Remark}
Let Assumption~\ref{ass:general_setting} be satisfied. 
Under the additional condition \ref{A5} (see Section~\ref{sec:mainresult}), from \cite[Lemma 17]{grae:popi:21}, for all $\varepsilon >0$, there exists $T_\varepsilon \in [0,T)$ such that for all $t\in[T_{\varepsilon},T)$ it holds 
\begin{equation}
    \Theta(T-t) \leq \phi \bigg( \dfrac{T-t}{(1+\varepsilon)\etamax} \bigg).
\end{equation}
In particular, 
under \ref{A5} and Assumption~\ref{ass:approximation}, 
it holds for every $\Delta \in (0,T-\max(\tau,T_\varepsilon)]$ that 
\begin{equation} \label{eq:a_priori_estim_term_cond_bis}
\| \xi_{T-\Delta} \|_{\infty} \leq \Theta(\Delta) + C\Delta^\alpha \leq \phi \left( \dfrac{\Delta}{(1+\varepsilon)\etamax}\right) + C \Delta^\alpha.
\end{equation}
Moreover, note that \eqref{eq:asymp_beha_Y} in Section~\ref{sec:mainresult} suggests the choice 
$ \xi_t = \phi((T-t)/\eta_t)$, $t\in[0,T)$
(see also the discussion below Theorem~\ref{thm:gene_case}). 
In this case, we directly obtain for all $\Delta \in (0,T]$ that 
\begin{equation} \label{eq:a_priori_estim_term_cond_ter}
\| \xi_{T-\Delta} \|_{\infty} \leq \phi \left( \dfrac{\Delta}{\etamax}\right).
\end{equation}
\end{Remark}

For every $\Delta \in (0,T-\tau]$ let $(Y^\Delta,Z^\Delta)$ be the solution of \eqref{eq:sing_BSDE} on $[0,T-\Delta]$ with the terminal condition $Y^\Delta_{T-\Delta}=\xi_{T-\Delta}$. Using the Conditions \ref{A1} and \ref{A3}, from classical a priori results for monotone BSDEs (see \cite[Section 5.3.1]{pard:rasc:14}) and 
\eqref{eq:cond_term} 
we get that there exists $C_1 \in [1,\infty)$ such that for all $\Delta\in (0,T-\tau]$ and for any $0\leq t \leq T-\Delta$ it holds 
\begin{align} \nonumber 
\bE_t\bigg[\sup_{r\in [t,T-\Delta]}|Y_r-Y_r^\Delta|^2+\int_t^{T-\Delta}\|Z_r-Z_r^\Delta\|^2dr \bigg] 
& \le C_1 \, \bE_t \big[ |Y_{T-\Delta}-\xi_{T-\Delta}|^2 \big]\\  \label{eq:estim_approx_BSDE}
& \le C_1 C \Delta^{2\alpha}.
\end{align}

Now for every $0< \Delta \leq T-\tau$ and $h>0$, let $(Y^{h,\Delta}, Z^{h,\Delta})$ be a time discretization of $(Y^\Delta,Z^\Delta)$ ($h$ being the step size). 
We need an estimate of the error $Y^\Delta-Y^{h,\Delta}$. Since the BSDE \eqref{eq:sing_BSDE} is monotone and if $f$ grows at most polynomially, we could use the results of Lionnet et al. \cite{lion:dosr:szpr:15}. Unfortunately, the authors are not explicit about the dependence of the constants on the parameters. Since $Y$ explodes at the time $T$, the value of $\xi_{T-\Delta}$ may be very large when $\Delta$ is small.

\subsection{Error analysis for the implicit Euler scheme for \eqref{eq:sing_BSDE} with bounded terminal condition}\label{subsec:imp_Euler_reg}

In this subsection we consider the BSDE \eqref{eq:sing_BSDE} but with a {\it bounded} and non-negative $\mathcal F_{T}$-measurable random condition $\xi$. We set $a_t=1/\eta_t$, $t\in [0,T]$, and consider the BSDE
\begin{equation}\label{eq:reg_BSDE}
    dY_t = - \left( a_t f(Y_t) + \lambda_t  \right) dt  + Z_tdW_t, \qquad t\in [0,T], \qquad Y_{T}=\xi.
\end{equation}
We first present bounds for the solution components $Y$ and $Z$ that we use in Proposition \ref{prop:error_imp_euler} below.

\begin{Lemma}\label{lem:ex_uni_ub_yz}
    Assume that Conditions {\rm \ref{A1}} and {\rm \ref{A3}} are satisfied. Then there exists a unique solution $(Y,Z)\in \mathbb S^2(0,T) $ of \eqref{eq:reg_BSDE}. The solution satisfies for all $t\in [0,T]$ a.s.\ that
    \begin{equation}
        \begin{split}
            0\le Y_t&\le \|\xi\|_\infty + T\lambdamax,\\
            \bE_t\bigg[\int_{t}^T Z_s^2ds \bigg] &\le \|\xi\|_\infty^2 +2(T-t)(\|\xi\|_\infty+T\lambdamax)\lambdamax.
        \end{split}
    \end{equation}
\end{Lemma}
\begin{proof}
    Existence and uniqueness follow from \cite[Proposition 5.24]{pard:rasc:14}. The lower bound comes from the comparison principle for BSDEs (see \cite[Proposition 5.33]{pard:rasc:14}). Then taking the conditional expectation w.r.t. $\mathcal F_t$ in \eqref{eq:reg_BSDE} gives the upper bound for $Y$. 

    To establish the bound for $Z$, we apply It\^o's formula to obtain for all $t\in [0,T]$ that 
    $$
    Y_t^2+\int_t^TZ_s^2ds = \xi^2 + 2\int_t^TY_s(a_sf(Y_s)+\lambda_s)ds-2\int_t^TZ_sY_sdW_s.
    $$
    The facts that $Z\in \mathbb H^2(0,T)$ and that $Y$ is bounded ensure that $\bE_t[\int_t^TZ_sY_sdW_s]=0$ for all $t\in[0,T]$. This, the fact that $f$ is non-increasing with $f(0)=0$ and the upper bound for $Y$ imply for all $t\in[0,T]$ that
    $$
    \bE_t\bigg[\int_t^T Z_s^2 ds \bigg] \le \|\xi\|^2_\infty+2(T-t)(\|\xi\|_\infty+T\lambdamax)\lambdamax.
    $$
    This completes the proof.
\end{proof}

The main aim of this subsection is to provide a complete error analysis of the implicit Euler scheme associated with \eqref{eq:reg_BSDE}. While such analyses have already been conducted in the literature -- often in more general settings (see, e.g., \cite{gobe:lemo:wari:05,gobe:lemo:wari:06,lion:dosr:szpr:15}) -- we include a detailed treatment here to keep the paper self-contained. Moreover, for our subsequent application, it is crucial to precisely quantify the influence of the norm of the terminal condition on the approximation error.

To this end, let $\pi=\{t_0,t_1,\ldots, t_N=T\}\subset [0,T]$ be an equidistant mesh with width $h=T/N$. Let $(\ol a_i)_{i=0,\ldots, N}$ and $(\ol \lambda_i)_{i=0,\ldots, N}$ be $(\mathcal{F}_{t_i})_{i=0,\ldots, N}$-adapted processes. We think of $(\ol a_i)_{i=0,\ldots, N}$ and $(\ol \lambda_i)_{i=0,\ldots, N}$ as discrete-time approximations of $(a_{t_i})_{i=0,\ldots, N}$ and $(\lambda_{t_i})_{i=0,\ldots, N}$. For this reason we assume for all $i\in\{0,\ldots,N\}$ that $0\le \ol a_i \le 1/\etamin$ and $0\le \ol \lambda_i \le \lambdamax$. Moreover, let $\ol \xi $ be a non-negative $\mathcal F_{T}$-measurable random variable with $\ol \xi \le \|\xi\|_\infty$.

Then we consider the $(\mathcal{F}_{t_i})_{i=0,\ldots, N}$-adapted process $(\ol Y_i)_{i=0,\ldots,N}$ defined via backward recursion by
\begin{equation}\label{eq:impl_euler}
    \ol Y_i=\bE_{t_i}[\ol Y_{i+1}] + h\ol a_i f(\ol Y_i) +h\ol \lambda_i , \qquad i=0,\ldots,N-1, \qquad \ol Y_N=\ol \xi.
\end{equation}
Note that since for all $a\ge 0$ the function $F_a(y)= y- ha f(y)$, $y\in [0,\infty)$, is strictly increasing (because of 
\ref{A3}), is equal to $0$ at $0$, satisfies $\lim_{y\to \infty}F_a(y)=\infty$ 
and since $\ol a_i$, $\ol \xi$ and $\ol \lambda_i$ are non-negative 
it follows by induction that there exists a unique non-negative solution $\ol Y_i$ in~\eqref{eq:impl_euler} for all $i=0,\ldots,N-1$. 

\begin{Proposition}\label{prop:error_imp_euler}
Assume that the Conditions {\rm \ref{A1}} and {\rm \ref{A3}} are satisfied and let $K=\|\xi\|_\infty + T \lambdamax$. Then it holds for all $i\in \{0,\ldots,N\}$ a.s.\ that
    \begin{equation*}
\begin{split}
|Y_{t_i}-\ol Y_i|&\leq  
\bE_{t_i}\bigg[|\xi - \ol \xi |+|f(K)|\int_{t_i}^T|a_s-\ol a_s|ds+\int_{t_i}^T|\lambda_s-\ol \lambda_s|ds \bigg]\\
& \quad +\bigg[\frac{|f'(K)|T}{2\etamin}\bigg(\frac{1}{\etamin}|f(K)|+\lambdamax\bigg)+\frac{\sup_{y\in [0,K]}|f''(y)|}{2\etamin}  (\|\xi\|^2_\infty+2(T-t_i)K\lambdamax)\bigg]h
\end{split}    
\end{equation*} 
where $\ol a_s = \ol a_k$ and $\ol \lambda_s = \ol \lambda_k$ on the interval $[t_k,t_{k+1})$. 
\end{Proposition}
\begin{proof}
We proceed by splitting the error into the one-step error and the stability error. To this end, for $a\ge 0$ let $F_a\colon [0,\infty) \to [0,\infty)$, $F_a(y)=y-ahf(y)$. As argued above, $F_a$ is a bijection. Now, for every $i\in \{0,\ldots,N-1\}$ let $ \ol \Gamma_i \colon L_+^\infty(\mathcal F_{t_{i+1}})\to L_+^\infty(\mathcal F_{t_{i}})$ satisfy for all non-negative, bounded, $\mathcal F_{t_{i+1}}$-measurable random variables $\mathcal Y$ that
$$
\ol \Gamma_i(\mathcal Y)=F^{-1}_{\ol a_i}(\bE_{t_i}[\mathcal Y]+ h\ol \lambda_i).
$$
Note that
$$
\ol \Gamma_i(\mathcal Y)=\bE_{t_i}[\mathcal Y]+ h(\ol a_if(\ol \Gamma_i(\mathcal Y)) + \ol \lambda_i) .
$$
Moreover, for every $i\in \{0,\ldots,N-1\}$ let $\Gamma_{t,i}\colon L_+^\infty(\mathcal F_{t_{i+1}})\to L_+^\infty(\mathcal F_{t})$ be the solution operator of the BSDE \eqref{eq:reg_BSDE} on $[t_i,t_{i+1}]$, that is, it holds for all non-negative, bounded, $\mathcal F_{t_{i+1}}$-measurable random variables $\mathcal Y$ that
$$
\Gamma_{t,i}(\mathcal Y)=\bE_{t}\left[\mathcal Y + \int_{t}^{t_{i+1}} (a_s f(\Gamma_{s,i}(\mathcal Y)) +\lambda_s) ds \right], \qquad t\in [t_i,t_{i+1}].
$$
Well-posedness of  $\Gamma_{t,i}$ follows from Lemma \ref{lem:ex_uni_ub_yz}.
Then we have for all $i\in \{0,\ldots,N-1\}$ that
\begin{equation}\label{eq:error_decomp}
    |Y_{t_i}-\ol Y_i|=|\Gamma_{t_i,i}(Y_{t_{i+1}})-\ol \Gamma_i(\ol Y_{i+1})|
\le |\Gamma_{t_i,i}(Y_{t_{i+1}})-\ol \Gamma_i(Y_{t_{i+1}})|+|\ol \Gamma_i(Y_{t_{i+1}})-\ol \Gamma_i(\ol Y_{i+1})| .
\end{equation}
The first term is the local error and the second the stability error. 

Let us first consider the stability error. Note that $F_a'(y)=1-ahf'(y) \ge 1$. This implies that $\frac{d}{dy}F_a^{-1}(y)\le 1$. From this it follows that $F_a^{-1}$ has a Lipschitz constant smaller than or equal to $1$ and hence we get for all $i\in \{0,\ldots,N-1\}$ that 
\begin{equation}\label{eq:stab_error}
    |\ol \Gamma_i(Y_{t_{i+1}})-\ol \Gamma_i(\ol Y_{i+1})|\le |\bE_{t_i}[Y_{t_{i+1}}-\ol Y_{i+1}]|.
\end{equation}

Now we consider the one-step error. Note that for all $i\in \{0,\ldots,N-1\}$ we have
\begin{align*}
\Gamma_{t_i,i}(Y_{t_{i+1}})-\ol \Gamma_i(Y_{t_{i+1}})
&=
\bE_{t_i}\left[\int_{t_i}^{t_{i+1}}a_s f(Y_s) ds - h\ol a_i f (\ol \Gamma_i(Y_{t_{i+1}})) \right]\\
& \quad +  \bE_{t_i}\left[\int_{t_i}^{t_{i+1}} \lambda_s ds \right] - h \ol \lambda_i \\
&=\bE_{t_i}\left[\int_{t_i}^{t_{i+1}} (a_s f(Y_s) - \ol a_{i} f(Y_{t_i})) ds\right]\\
&\quad
+ h\ol a_{i} \left( f(Y_{t_i}) - f(  \ol \Gamma_i(Y_{t_{i+1}})) \right) +  \bE_{t_i}\left[\int_{t_i}^{t_{i+1}} ( \lambda_s - \ol \lambda_i ) ds \right] .
\end{align*}
This implies for all $i\in \{0,\ldots,N-1\}$ that
\begin{equation}\label{eq:1133}
\begin{split}
&\left(\Gamma_{t_i,i}(Y_{t_{i+1}})-\ol \Gamma_i(Y_{t_{i+1}})\right)^2
+ h^2\ol a_{i}^{2} \left( f(\Gamma_{t_i,i}(Y_{t_{i+1}}))- f ( \ol \Gamma_i(Y_{t_{i+1}})) \right)^2\\
&\quad =\left(  \bE_{t_i}\left[\int_{t_i}^{t_{i+1}} (a_s f(Y_s) -\ol a_{i} f(Y_{t_i})) ds\right]
+  \bE_{t_i}\left[\int_{t_i}^{t_{i+1}} ( \lambda_s - \ol \lambda_i ) ds \right]  \right)^2\\
&\qquad
+2 h\ol a_{i}( f(\Gamma_{t_i,i}(Y_{t_{i+1}}))- f (\ol \Gamma_i(Y_{t_{i+1}})))\left(\Gamma_{t_i,i}(Y_{t_{i+1}})-\ol \Gamma_i(Y_{t_{i+1}})\right).
\end{split}
\end{equation}
For every $i\in \{0,\ldots,N-1\}$ let 
$$
R_i= \left|   \bE_{t_i}\left[\int_{t_i}^{t_{i+1}} (a_s f(Y_s) -\ol a_{i} f(Y_{t_i})) ds\right]
+  \bE_{t_i}\left[\int_{t_i}^{t_{i+1}} ( \lambda_s - \ol \lambda_i ) ds \right] \right|.
$$
Then 
the fact that $y\mapsto f(y)$ is non-increasing 
(Condition \ref{A3}) and \eqref{eq:1133} yield 
for all $i\in \{0,\ldots,N-1\}$ that 
\begin{equation*}
\left|\Gamma_{t_i,i}(Y_{t_{i+1}})-\ol \Gamma_i(Y_{t_{i+1}})\right| \le R_i.
\end{equation*}
Combining this with \eqref{eq:stab_error} and \eqref{eq:error_decomp} implies for every $i\in \{0,\ldots,N-1\}$ that 
$$ |Y_{t_i}-\ol Y_i|\leq  \bE_{t_i} |Y_{t_{i+1}}-\ol Y_{i+1} | + R_i$$
and thus by induction
\begin{equation}\label{eq:error_y_Ri}
    |Y_{t_i}-\ol Y_i|\leq  \bE_{t_i} |\xi - \ol \xi | + \sum_{k=i}^{N-1}\bE_{t_i}[R_k].
\end{equation}
Note that for all $i\in \{0,\ldots,N-1\}$ it holds 
\begin{equation}\label{eq:decomp_Ri}
\begin{split}
R_i &\leq  \left|   \bE_{t_i}\left[\int_{t_i}^{t_{i+1}} (a_s f(Y_s) -\ol a_{i} f(Y_{t_i})) ds\right]  \right| 
+ \left|\bE_{t_i}\left[\int_{t_i}^{t_{i+1}}  (\lambda_s - \ol \lambda_i )ds \right]\right|\\
&\le  \left|   \bE_{t_i}\left[\int_{t_i}^{t_{i+1}} (a_s-\ol a_i) f(Y_s)ds \right]  \right|  + \left|   \bE_{t_i}\left[\int_{t_i}^{t_{i+1}}\ol a_{i} (f(Y_s)-f(Y_{t_i}))ds\right]  \right| \\
&\quad
+ \left|\bE_{t_i}\left[\int_{t_i}^{t_{i+1}}  (\lambda_s - \ol \lambda_i )ds \right]\right| .
\end{split}    
\end{equation}
Recall that $K=\|\xi\|_\infty + T \lambdamax$.
The fact that by Lemma \ref{lem:ex_uni_ub_yz} it holds a.s.\ for all $t\in [0,T]$ that $Y_t\le K$ together with the monotonicity of $f$ implies that $|f(Y_t)| \leq -f (K)$. Therefore we have for all $i\in \{0,\ldots,N-1\}$ that 
\begin{equation}\label{eq:aux_Ri_1}
    \left|   \bE_{t_i}\left[\int_{t_i}^{t_{i+1}} (a_s-\ol a_i) f(Y_s)ds \right]  \right|\le   - f (K)\bE_{t_i}\left[\int_{t_i}^{t_{i+1}} |a_s-\ol a_i| ds \right] .
\end{equation}
Moreover, note that by It\^o's formula, together with the boundedness results in Lemma~\ref{lem:ex_uni_ub_yz}, which ensure that the stochastic integral vanishes in expectation, we have for all $i\in \{0,\ldots,N-1\}$ that 
\begin{equation*}
\begin{split}
    & \left|   \bE_{t_i}\left[\int_{t_i}^{t_{i+1}}\ol a_{i} (f(Y_s)-f(Y_{t_i}))ds\right]  \right| \\
    & =\ol a_{i}\left|   \bE_{t_i}\left[\int_{t_i}^{t_{i+1}}\int_{t_i}^s \Big(-f'(Y_r)(a_rf(Y_r)+\lambda_r)+\frac{1}{2}f''(Y_r)Z_r^2 \Big) \,dr \, ds\right]  \right|.
\end{split}
\end{equation*}
This together with the fact that by Lemma \ref{lem:ex_uni_ub_yz} it holds a.s.\ for all $t\in [0,T]$ that $Y_t\le K$, Assumption \ref{A1}, Assumption \ref{A3} and Fubini's theorem implies for all $i\in \{0,\ldots,N-1\}$ that 
\begin{equation*}
\begin{split}
    \left|   \bE_{t_i}\left[\int_{t_i}^{t_{i+1}}\ol a_{i} (f(Y_s)-f(Y_{t_i}))ds\right]  \right|
    &\le \frac{1}{2\etamin}|f'(K)|\left(\frac{1}{\etamin}|f(K)|+\lambdamax\right)h^2\\
    &\quad +\frac{\sup_{y\in [0,K]}|f''(y)|}{2\etamin}  \bE_{t_i}\left[\int_{t_i}^{t_{i+1}} (t_{i+1}-r)Z_r^2 dr \right] \\
    &\le \frac{1}{2\etamin}|f'(K)|\left(\frac{1}{\etamin}|f(K)|+\lambdamax\right)h^2\\
    &\quad +\frac{\sup_{y\in [0,K]}|f''(y)|}{2\etamin}  \bE_{t_i}\left[\int_{t_i}^{t_{i+1}} Z_r^2 dr \right]h .
\end{split}
\end{equation*}
Combining this with \eqref{eq:error_y_Ri}, \eqref{eq:decomp_Ri} and \eqref{eq:aux_Ri_1} yields for all $i\in \{0,\ldots,N-1\}$ that 
\begin{equation*}
\begin{split}
|Y_{t_i}-\ol Y_i|&\leq  
\bE_{t_i}\left[|\xi - \ol \xi |+|f(K)|\int_{t_i}^T|a_s-\ol a_s|ds+\int_{t_i}^T|\lambda_s-\ol \lambda_s|ds \right]\\
&\quad +\frac{1}{2\etamin}|f'(K)|\left(\frac{1}{\etamin}|f(K)|+\lambdamax\right)Th+\frac{\sup_{y\in [0,K]}|f''(y)|}{2\etamin}  \bE_{t_i}\left[\int_{t_i}^{T} Z_r^2 dr \right]h.
\end{split}    
\end{equation*}
Applying Lemma \ref{lem:ex_uni_ub_yz} once more we obtain for all $i\in \{0,\ldots,N-1\}$ that 
\begin{equation*}
\begin{split}
|Y_{t_i}-\ol Y_i|&\leq  
\bE_{t_i}\left[|\xi - \ol \xi |+|f(K)|\int_{t_i}^T|a_s-\ol a_s|ds+\int_{t_i}^T|\lambda_s-\ol \lambda_s|ds \right]\\
& \quad +\frac{1}{2\etamin}|f'(K)|\left(\frac{1}{\etamin}|f(K)|+\lambdamax\right)Th \\
& \quad +\frac{\sup_{y\in [0,K]}|f''(y)|}{2\etamin}  \left[ \|\xi\|^2_\infty+2(T-t_i)K\lambdamax\right]h.
\end{split}    
\end{equation*} 
This completes the proof.
\end{proof}

\subsection{Error analysis for the implicit Euler scheme for \eqref{eq:sing_BSDE} with singular terminal condition}\label{subsec:imp_Euler_sing}

In this subsection we come back to the BSDE \eqref{eq:sing_BSDE} with the singular terminal condition $\xi=\infty$. We assume that Assumption \ref{ass:approximation} holds. 
Therefore, we can apply the results of Section~\ref{subsec:imp_Euler_reg} for the time discretization of $Y^\Delta$ starting from $\xi_{T-\Delta}$ on every time interval $[0,T-\Delta]$ with $\Delta\in (0,T-\tau]$.
In the following we use the setting and notation of Section~\ref{subsec:imp_Euler_reg} with the exception that the time horizon $T$ is now given by $T-\Delta$. In particular, $\ol \xi$ is $\mathcal F_{T-\Delta}$-measurable.

\begin{Corollary}\label{coro:error_imp_euler}
Assume that Conditions {\rm \ref{A1}} and {\rm \ref{A3}} are satisfied. Let $\Delta\in (0,T-\tau]$ and let $K=\|\xi_{T-\Delta}\|_\infty + T \lambdamax$. Then it holds for all $i\in \{0,\ldots,N\}$ a.s.\ that
\begin{equation*}
\begin{split}
& |Y^\Delta_{t_i}-\ol Y_i| \\
& \leq  
\bE_{t_i}\bigg[| \xi_{T-\Delta}- \ol \xi |+|f(K)|\int_{t_i}^T|a_s-\ol a_s|ds+\int_{t_i}^T|\lambda_s-\ol \lambda_s|ds \bigg]\\
&\quad + \bigg[\frac{|f'(K)|T}{2\etamin}\bigg(\frac{1}{\etamin}|f(K)|+\lambdamax \bigg) + \frac{\sup_{y\in [0,K]}|f''(y)|}{2\etamin}  (\| \xi_{T-\Delta}\|^2_\infty+2(T-\Delta-t_i)K\lambdamax)\bigg] h.
\end{split}    
\end{equation*} 
\end{Corollary}
\begin{proof}
    This is an immediate consequence of Proposition \ref{prop:error_imp_euler}.
\end{proof}

From Corollary~\ref{coro:error_imp_euler} and \eqref{eq:estim_approx_BSDE}, we derive an estimate of the error for $Y$: For all $i\in \{0,\ldots,N\}$ we have a.s.\ that 
\begin{equation}\label{eq:error_estimate_Y}
\begin{split}
|Y_{t_i}-\ol Y_i|&\leq C_1 C \Delta^\alpha + 
\bE_{t_i} \bigg[|  \xi_{T-\Delta}- \ol \xi |+|f(K)|\int_{t_i}^T|a_s-\ol a_s|ds+\int_{t_i}^T|\lambda_s-\ol \lambda_s|ds \bigg] \\
&\quad +\bigg[\frac{|f'(K)|T}{2\etamin}\left(\frac{1}{\etamin}|f(K)|+\lambdamax\right)+\frac{\sup_{y\in [0,K]}|f''(y)|}{2\etamin}  ( K^2 + (T\lambdamax)^2)\bigg]h.
\end{split}
\end{equation} 
From \eqref{eq:a_priori_estim_term_cond}, the value of $K$ is bounded by 
$\Theta \left(\Delta \right) + C\Delta^\alpha+ T\lambdamax$ (other estimates are possible with \eqref{eq:a_priori_estim_term_cond_bis} or \eqref{eq:a_priori_estim_term_cond_ter} under stronger conditions). 

Now let us assume that 
\begin{Assumption}\label{ass:regularity_approx}
There exist two constants $C_2 > 0$ and $\beta > 0$ and a function $\Phi\colon (0,T] \to [0,\infty)$ such that for any $h>0$ and $0< \Delta\leq T-\tau$ and all $0\leq i\leq N$ it holds a.s. 
\begin{equation} \label{eq:conv_rate_Euler}
\bE_{t_i}\bigg[\int_{t_i}^T|a_s-\ol a_s|ds+\int_{t_i}^T|\lambda_s-\ol \lambda_s|ds \bigg] \leq C_2 h^{\beta}
\end{equation}
and 
$$\bE_{t_i}\big[| \xi_{T-\Delta}- \ol \xi |\big] \leq C_2 \Phi(\Delta) h^{\beta}.$$
\end{Assumption}
If $(1/\eta,\lambda)$ is the solution of an SDE with Lipschitz-continuous coefficients, it is well known that the Euler scheme gives such an approximation \eqref{eq:conv_rate_Euler} with $\beta=1/2$ (see \cite{kloe:plat:92}). And the constant $C_2$ depends on the bounds on the coefficients $\eta$ and $\lambda$. For the difference $\xi_{T-\Delta}-\ol \xi$, we need to take into account the norm of $\xi_{T-\Delta}$, which depends on $\Delta$; this is the reason for the presence of $\Phi(\Delta)$ in the estimate.

Hence we can now derive our main result for the approximation error of $Y$:
\begin{Theorem} \label{thm:error}
    Suppose that Assumptions \ref{ass:general_setting}, \ref{ass:approximation} and \ref{ass:regularity_approx} hold. Then there exist a constant $C_3>0$ and two functions $\Psi_1$ and $\Psi_2$, which only depend on $f$, $\eta$, $\lambda$, such that for any $h>0$ and $0< \Delta \leq T-\tau$ and for all $i\in \{0,\ldots,N\}$ it holds a.s. 
    \begin{equation*}
|Y_{t_i}-\ol Y_i|\leq  C_3 \big[ \Delta^\alpha + h^\beta + \Psi_1(\Delta) h^\beta + \Psi_2(\Delta) h\big].  
\end{equation*} 
\end{Theorem}
\begin{proof}
We use our previous result \eqref{eq:error_estimate_Y} with $K=\Theta \left(\Delta \right) + C\Delta^\alpha+ T\lambdamax$ and 
\begin{align*}
   \Psi_1(\Delta) & = |f(K)| + \Phi(\Delta),\\
  \Psi_2(\Delta) & = \frac{|f'(K)|T}{2\etamin}\left(\frac{1}{\etamin}|f(K)|+\lambdamax\right)+\frac{\sup_{y\in [0,K]}|f''(y)|}{2\etamin}  ( K^2 + (T\lambdamax)^2)  .
\end{align*}
This proves the claim. 
\end{proof}

Obviously $\Psi_1$ and $\Psi_2$ tend to infinity when $\Delta$ goes to zero. Hence there is a balance to find between $\Delta$ and $h$. 

\section{Expansion of $Y$ near $T$} \label{sect:approx}

We still assume that Conditions \ref{A1} to \ref{A4} hold. The goal is to find sufficient assumptions on $f$ such that \eqref{eq:cond_term} holds.

\subsection{Main result}\label{sec:mainresult}

Recall that $G$, the antiderivative of $1/f$, is positive and we consider $\phi\colon (0,\infty)\to (0,\infty)$ to be the inverse of $G$. Also evoke that $\phi$ solves the ODE $\phi' = f\circ \phi$ with $\lim_{x\to 0} \phi(x)=\infty$ and note that $\phi'$ is negative and $\phi''$ positive by Lemma~\ref{lem:ineq_phi} in the appendix. In this section, our aim is to show that 
a.s. 
 \begin{equation} \label{eq:asymp_beha_Y}
Y_t = \phi \left(\dfrac{T-t}{\eta_t} \right) - \phi'\left( \dfrac{T-t}{\eta_t} \right) H_t,\quad \forall t \in [0,T),
\end{equation}
where $H$ is the solution of a BSDE, obtained via a fixed-point argument, with a suitable behavior at the time $T$. 

Assume for a while that the relation~\eqref{eq:asymp_beha_Y} holds. 
Then with It\^o's formula we deduce
that on $[0,T)$ the process $H$ has the dynamics 
\begin{align} \nonumber 
dH_t & = \frac{1}{\eta_t \phi'\left( A_t \right) } \left[ f( \phi(A_t) - \phi'(A_t) H_t) - f(\phi(A_t)) + f'(\phi(A_t)) \phi'(A_t) H_t \right] dt \\ \nonumber 
& \quad + \left[ \frac{1}{\phi'\left( A_t \right) } \lambda_t  +A_t \left( - \dfrac{b^\eta_t }{\eta_t} + \left(  \dfrac{\sigma^\eta_t}{\eta_t} \right)^2 \left( 1 - \dfrac{\kappa^1_t }{2} \right) \right)\right] dt  - \mu_t  H_t dt - \kappa^1_t  \dfrac{\sigma^\eta_t }{\eta_t}  Z^H_t dt + Z^H_t dW_t   \\  \label{eq:BSDE_H}& 
=: - F^H(t,H_t,Z^H_t) dt + Z^H_t dW_t, 
\end{align}
where for all $x \in (0,\infty)$, $t \in [0,T)$ and $i\in\{0,1,2\}$ we have 
\begin{equation*}
    A_t = \dfrac{T-t}{\eta_t}, 
    \quad \kappa^i(x) = -\dfrac{\phi^{(i+1)}(x)}{\phi^{(i)}(x)} x
\end{equation*}
and 
\begin{equation} \label{eq:coeff_BSDE_H}
\kappa^i_t = \kappa^i(A_t)=- \dfrac{\phi^{(i+1)} \left(A_t \right)  }{\phi^{(i)}\left(A_t \right)}A_t, \quad \mu_t = \kappa^1_t \dfrac{ b^\eta_t}{\eta_t}  + \kappa^1_t  \left(\dfrac{\kappa^2_t}{2}-1\right) \left( \dfrac{\sigma^\eta_t}{ \eta_t }\right)^2 .
\end{equation}
We have the following relation between the functions $x\mapsto \kappa^i(x)$, which we prove in the appendix: 
\begin{Lemma} \label{lem:ineg_kappa_i_x}
Assume that Conditions \ref{A1} to \ref{A4} are satisfied. Then 
\begin{equation} \label{ineq:kappa_i}
\forall x > 0\colon \quad 0\leq \kappa^0(x) \leq  \kappa^1(x) \leq \kappa^2(x).
\end{equation}
\end{Lemma}
We further suppose that:
\begin{enumerate}[label={\rm \textbf{(A\arabic*)}}]
\setcounter{enumi}{4}
\item \label{A5} {\it There exists $\epsilon > 0$ such that the function $(0,\infty) \ni x\mapsto \kappa^2(x)$ is bounded on the interval $(0,\epsilon)$.
}
\end{enumerate}
We claim that:
\begin{Lemma} \label{lem:ineg_kappa_i}
Assume that Conditions \ref{A1} to \ref{A5} are satisfied. Then 
the processes $\kappa^0$, $\kappa^1$, $\kappa^2$ and $\mu$ defined by \eqref{eq:coeff_BSDE_H} 
are bounded on $[0,T)\times \Omega$.
\end{Lemma}
The proof is set up in the appendix. 
Denote by $\kappa^\star$ and $\mu^\star$ the upper bound of $\kappa^1$ and $|\mu|$, respectively.
It then follows that 
the process 
$$\zeta_t = \exp \left(\int_0^t \mu_s ds \right), \quad t\in [0,T],$$
is bounded on $[0,T]\times \Omega$ with the upper bound $\zeta^\star=\exp(T \mu^\star)$
and the lower bound 
$\exp(-T\mu^\star) = \zeta_\star > 0$.
The next quantity is crucial in the rest of this part:
\begin{equation} \label{eq:def_const_K}
K =   \zeta^\star \etamin \left[ \lambda^\star +\dfrac{1}{2} \left\| \dfrac{b^\eta}{\eta} \right\|_\infty + \dfrac{1}{2} \left(   \dfrac{\kappa^\star}{2} + 1 \right) \left\| \dfrac{\sigma^\eta}{\eta} \right\|_\infty^2 \right].
\end{equation}
We define for $x\geq 0$ the quantities
\begin{equation} \label{eq:rate_conv}
\varpi(x) = \int_0^{x}  \frac{ 1 }{-\phi'\left( z \right) } dz,\quad  \vartheta (x) = \max\big(  \varpi(x), x^2 \big) .
\end{equation}
The function $\vartheta$ gives the rate of convergence of $t \mapsto H_t$ to zero when $t$ goes to $T$. 
Note that since $\phi'$ is a 
negative and increasing function (cf.\ Lemma~\ref{lem:ineq_phi}) it holds for all $x>0$ that 
$0\leq \varpi(x) \leq  \dfrac{ x }{-\phi'(x) }.$ 
We thus have for all $x>0$ that 
\begin{equation} \label{eq:estim_vartheta}
0\leq \vartheta(x) \leq x  \max \bigg( \dfrac{ 1 }{-\phi'\left(x\right) } , x \bigg).
\end{equation}
Since $\phi' = f \circ \phi$ with $\lim_{x\to 0} \phi(x) = +\infty$, we have $\displaystyle  \lim_{x\to 0} \vartheta(x) = 0.$
Some properties of $\vartheta$ are discussed in Section~\ref{ssect:vartheta}. 
Before presenting our last condition, we note the following technical result, with the proof given in the appendix.
\begin{Lemma}\label{lem:lastcondwelldef}
    Assume that Conditions \ref{A3} to \ref{A5} are satisfied. Let $\eta^\sharp \ge 1$ and $\varsigma>0$. Then 
    there exists $\widetilde \epsilon > 0$ such that 
    $$\forall x \in (0,\widetilde \epsilon)\colon \quad \phi(x) + \varsigma \phi'(x) \vartheta(\eta^\sharp x) > 0.$$
\end{Lemma}
Remark that $\phi(x) - \varsigma \phi'(x) \vartheta(\eta^\sharp x) > 0$ in the setting of Lemma~\ref{lem:lastcondwelldef}  because $-\phi'$ is non-negative.
Let us now formulate our last assumption. 
\begin{enumerate}[label={\rm \textbf{(A\arabic*)}}]
\setcounter{enumi}{5}
\item \label{A6} For some $\varsigma > 2K / \zeta_\star$, with $\eta^\sharp = \etamax / \etamin$, there exists $\widehat \epsilon > 0$ such that  the non-negative function
$$x \mapsto   \big| f'\big(\phi(x) \pm\varsigma \phi'(x) \vartheta( \eta^\sharp x)\big) - f'(\phi(x)) \big|$$ 
is bounded on $(0,\widehat \epsilon)$ by some non-increasing function $\Psi$ such that $\Psi$ is integrable on $(0,\widehat \epsilon)$ and $\vartheta( \eta^\sharp \cdot) \Psi$ is bounded on $(0,\widehat \epsilon)$. 
\end{enumerate}
Let us emphasize that under our setting, using the previous Lemma~\ref{lem:lastcondwelldef}, the terms inside $f'$ in this condition are positive and thus everything is well-defined when $\widehat \epsilon \leq \widetilde \epsilon$, with $\widetilde \epsilon$ being given in Lemma \ref{lem:lastcondwelldef}. Hence in the rest of this section, with an abuse of notation, we define $\epsilon>0$ as the minimum between the coefficient $\epsilon$ given in \ref{A5} and the coefficient $\widehat \epsilon$ given in \ref{A6}.  
And the properties of \ref{A5} and of \ref{A6} hold on this smaller interval $(0,\epsilon)$.

\begin{Remark}[On Conditions \ref{A5} and \ref{A6}]
Note that Conditions {\rm \ref{A5}} and {\rm \ref{A6}} just depend on the behavior of $f$ and $f'$ on a neighborhood of $\infty$.  
To see this for Condition~{\rm \ref{A5}}, we recall some results from \cite[Lemmata 11 and 12]{grae:popi:21}. 
First, boundedness of $\kappa^1$ is equivalent to the existence of two constants $\delta > 0$ and $R > 0$ with $G(R) = \epsilon$ such that $x \mapsto \displaystyle  \left( \int_x^\infty \dfrac{1}{-f(y)} dy \right)^{-\delta}$ is convex on $[R,+\infty)$. 
Second, additional boundedness of $\kappa^2$ is equivalent to: $(-f)\circ \psi^\delta$ is also increasing and concave on $[R,+\infty)$, where $\psi^\delta$ is the increasing and concave function $\psi^\delta \colon (0,\infty) \to (0,\infty)$, $ x \mapsto \phi (x^{-1/\delta})$. 
In a similar way, we can consider $y \mapsto | f'[y \pm \varsigma  f(y) \vartheta(\eta^\sharp G(y))] - f'(y) |$, bounded by $\Psi\circ G$ integrable on some interval $[R,+\infty)$, to see that Condition {\rm \ref{A6}} only depends on the behavior of $f'$ on a neighborhood of $\infty$. 
%


\end{Remark}
In the following lemma we present a sufficient condition for \ref{A6}, with the proof given in the appendix.
\begin{Lemma} \label{lem:checking_A6}
Assume that Conditions \ref{A3} to \ref{A5} are satisfied and that there exists $\widetilde \epsilon >0$ such that $\vartheta(x)=x^2$ for all $x \in [0,\widetilde\epsilon)$. 
Furthermore, suppose that there exist $R>0$ and $\varsigma >\frac{2K}{\zeta_*} (\eta^\sharp)^2$ such that the functions $y \mapsto  \dfrac{f^{(2)}[y \pm \varsigma f(y) (G(y))^{2}]}{f^{(2)}(y)}$ are bounded on $(R,+\infty)$. 
Then {\rm \ref{A6}} holds with $\Psi$ a constant function.  
\end{Lemma}

Define 
$$\theta(x) = \vartheta \bigg( \dfrac{x}{\etamin} \bigg), \quad x \ge 0,$$ 
and consider for any $\delta \in (0,T]$ the space  
\begin{equation} \label{eq:def_cH_delta} 
\mathcal H^{\delta,\theta}:=\big\{H\in L^\infty(\Omega;C([T-\delta, T];\mathbb R)),\ H \mbox{ adapted}:\|H\|_{\mathcal H^{\delta,\theta}}<+\infty \big\}
\end{equation}
 endowed with the weighted norm 
\begin{equation} \label{eq:def_norm_cH_delta}
\|H\|_{\mathcal H^{\delta,\theta}}=\left\|\sup_{t\in{[T-\delta,T)}} \dfrac{|H_t|}{\theta(T-t)} \right\|_\infty.
\end{equation}
Our main result is the following. 

\begin{Theorem} \label{thm:gene_case}
Assume that Conditions {\rm \ref{A1}} to {\rm \ref{A6}} are satisfied. 
Then there exists $\delta > 0$ such that the BSDE~\eqref{eq:BSDE_H} with the terminal condition $0$ has a unique solution $(H,Z^{H})$ on $[0,T]$ with $H \in \mathcal H^{\delta,\theta} \cap L^\infty([0,T] \times \Omega)$,  
and $\int_{0}^\cdot Z^{H} dW$ is a BMO-martingale. Moreover, the relation~\eqref{eq:asymp_beha_Y} holds. 
\end{Theorem}

From this result we deduce that there exists $\delta>0$ such that for all $t \in [T-\delta,T)$ it holds  
$$Y_t = \phi(A_t ) - \phi'(A_t )H_t = \xi_t - \phi'(A_t)H_t ,$$
where $\xi_t:=\phi(A_t)$, $t\in [T-\delta,T)$. 
Therefore Condition \eqref{eq:cond_term} is true if we can control the term $\phi'(A_t )H_t$, $t\in [T-\delta,T)$. 
Evoke that 
\cite[Lemma 11 and Equation (38)]{grae:popi:21} 
shows that there is a constant $k>0$ such that 
for 
all $t \in (T-\min(\delta,\etamin \epsilon),T)$ it holds  
\begin{equation}\label{eq:bound_second_term}
| - \phi'(A_t) H_t | \leq  
- \phi'\bigg( \dfrac{T-t}{\etamax} \bigg)\vartheta \bigg( \dfrac{T-t}{\etamin} \bigg) \|H\|_{\mathcal H^{\delta,\theta}} 
\leq - (\eta^\sharp)^k \phi'\bigg( \dfrac{T-t}{\etamin} \bigg)\vartheta \bigg( \dfrac{T-t}{\etamin} \bigg) \|H\|_{\mathcal H^{\delta,\theta}}  .
\end{equation}
When $\vartheta = \varpi$ (faster convergence), then we have an upper bound of the form $O(T-t)$.  However when $\vartheta(x) = x^2$ (slower convergence), the convergence to zero of this quantity has to be proved on a case-by-case basis and is discussed for some examples in Section  \ref{ssect:examples}.

\subsection{Proof of Theorem \ref{thm:gene_case}}

This proof is split in several steps. 

\medskip

\noindent {\bf Step 1.} 
Here we observe that, by a transformation, $Z^H$ can be removed from the drift of the BSDE~\eqref{eq:BSDE_H}.  
Since $\mu$ and $\zeta$ are bounded processes, we can define a new probability measure $\mathbb Q$ equivalent to $\mathbb P$  with 
$$\dfrac{d\mathbb Q}{d \mathbb P}\bigg|_{\mathcal F_T} = \exp\left(\int_0^T  \kappa^1_t  \sigma^\eta_t / {\eta_t} dW_t-\frac{1}{2}\int_0^T\bigl( \kappa^1_t  \sigma^\eta_t /{\eta_t} \bigr)^2dt\right).$$

Suppose first that $(H,Z^H)$ is a solution of the BSDE~\eqref{eq:BSDE_H} in the sense of this theorem. 
Then $\widehat H = \zeta H$ and $\widehat Z = \zeta Z^H$ satisfy the BSDE 
\begin{equation}\label{eq:BSDEwidehat}
d\widehat H_t  = \mu_t \zeta_t H_t dt  - \zeta_t F^H(t,(\zeta_t)^{-1} \widehat H_t,  (\zeta_t)^{-1} \widehat Z_t) dt  +  \widehat Z_t  dW_t  = -\widehat F(t,\widehat H_t) dt + \widehat Z_t d\widehat  W_t
\end{equation}
with $\widehat H_T=0$ 
where 
\begin{align*}
\widehat F(t,h) & = \frac{-\zeta_t }{\eta_t \phi'(A_t)} \left[ f \left(\phi(A_t)-\phi'(A_t) \dfrac{h}{\zeta_t}\right) - f(\phi(A_t)) +f'( \phi(A_t)) \phi'(A_t) \dfrac{h}{\zeta_t} \right] \\
& \quad - \frac{ \zeta_t \lambda_t }{\phi'\left( A_t \right) } +A_t \zeta_t \left[  \dfrac{ b^\eta_t }{\eta_t}+ \left(\dfrac{\kappa^1_t}{2}-1\right) \left( \dfrac{ \sigma^\eta_t }{\eta_t} \right)^2 \right] 
\end{align*}
and
$$\widehat  W_t =W_t - \int_0^t \kappa^1_s  \dfrac{\sigma^\eta_s }{\eta_s}  ds. 
$$
Note that $\widehat W$ is a $\mathbb Q$-Brownian motion by Girsanov's theorem.
Moreover, note that $\mathbb Q$-a.s.\ for all $t \in [T-\delta,T]$ it holds $\lvert \widehat H_t \rvert \le \zeta^* \|H\|_{\mathcal H^{\delta,\theta}} \, \theta(T-t)$ and $\int_0^{\cdot} \widehat Z d \widehat W$ is a BMO-martingale. 

Conversely, if $(\widehat H, \widehat Z)$ 
is a solution of the BSDE~\eqref{eq:BSDEwidehat} in the sense of this theorem, then  
setting $H_t=(\zeta_t)^{-1} \widehat H_t$ and $Z^H_t=(\zeta_t)^{-1} \widehat Z_t$ for all $t \in [0,T]$ yields a solution of the BSDE~\eqref{eq:BSDE_H} with the desired properties and the same $\delta>0$. 


Therefore, to avoid additional notation, we now consider\footnote{With an abuse of notation.} 
in Step 2 to Step 6 of this proof 
the BSDE 
\begin{equation}\label{eq:BSDE_F} 
dH_t  = -F(t,H_t) dt + Z_t dW_t
\end{equation}
where 
\begin{align} \nonumber 
F(t,h) &= \frac{-\zeta_t }{\eta_t \phi'(A_t)} \left[ f \left(\phi(A_t)-\phi'(A_t) \dfrac{h}{\zeta_t}\right) - f(\phi(A_t)) +f'( \phi(A_t)) \phi'(A_t) \dfrac{h}{\zeta_t} \right] \\ \label{eq:gene_F}
 &  \quad - \frac{ \zeta_t \lambda_t }{\phi'\left( A_t \right) } +A_t \zeta_t \left[  \dfrac{ b^\eta_t }{\eta_t}+ \left(\dfrac{\kappa^1_t}{2}-1\right)  \left( \dfrac{ \sigma^\eta_t }{\eta_t} \right)^2  \right] .
\end{align}

\medskip

\noindent {\bf Step 2.} 
In this step, we obtain an estimate for 
$$\Gamma_t^0 := \mathbb E_t\bigg[\int_t^T F(s,0) ds \bigg], \quad t\in [0,T].$$
We have for all $t \in [0,T]$ that 
\begin{align*}
|\Gamma_t^0|  \leq   
\mathbb E_t\bigg[\int_t^T |F(s,0)| \,ds\bigg] 
& \leq   \lambda^\star  \zeta^\star\mathbb E_t\bigg[\int_t^T  \frac{ 1 }{-\phi'( A_s) } ds \bigg] \\
	& \quad + \zeta^\star \left[  \left\| \dfrac{b^\eta}{\eta} \right\|_\infty +  \left\|\dfrac{\kappa^1}{2}-1\right\|_\infty \left\| \dfrac{\sigma^\eta}{\eta} \right\|_\infty^2  \right]  \mathbb E_t\bigg[\int_t^T A_s ds\bigg] . 
\end{align*}
Note that since $\phi'$ is negative and increasing (cf.~Lemma~\ref{lem:ineq_phi}), we have for all $t\in [0,T]$ that   
$$0\leq \mathbb E_t\bigg[\int_t^T  \frac{ 1 }{-\phi'\left( A_s \right) } ds \bigg] 
\leq  \int_t^T  \frac{ 1 }{-\phi'\left( (T-s)/\etamin \right) } ds  \leq \etamin \varpi \left( \dfrac{T-t}{\etamin} \right) .$$ 
Moreover, we obtain for all $t\in[0,T]$ that 
$$0\leq \mathbb E_t\bigg[\int_t^T A_s ds \bigg]  
\leq  \int_t^T \dfrac{(T-s)}{\etamin} ds  = \dfrac{(T-t)^2 }{2\etamin} = \dfrac{\etamin}{ 2} \left( \dfrac{T-t }{\etamin} \right)^2.$$
Hence it holds for all $t\in[0,T]$ that 
$$|\Gamma_t^0| \leq  \zeta^\star \etamin \lambda^\star  \varpi\left(\dfrac{T-t}{\etamin} \right)  +  \dfrac{  \zeta^\star \etamin}{2} \left[  \left\| \dfrac{b^\eta}{\eta} \right\|_\infty +  \left\|\frac{\kappa^1}{2}-1\right\|_\infty \left\| \dfrac{\sigma^\eta}{\eta} \right\|_\infty^2  \right] \left(\dfrac{T-t}{\etamin} \right)^2 .$$
The definition of $K$ given by \eqref{eq:def_const_K} comes from this inequality and we obtain for all $t\in[0,T]$ that  
\begin{align} \label{eq:init_value}
|\Gamma_t^0| & \leq K \vartheta \left( \dfrac{T-t}{\etamin} \right).
\end{align}

\medskip

\noindent {\bf Step 3.} 
We next introduce a solution operator for the BSDE~\eqref{eq:BSDE_F} close to the time~$T$. 
We claim that: 
\begin{Lemma}  \label{lem:bounded_generator}
Assume that Conditions {\rm \ref{A1}} to {\rm \ref{A6}} are satisfied. Let $\delta \in (0, \etamin \epsilon)$ and let $H$ be a process such that a.s.\ for all $t \in [T-\delta,T]$ it holds $|H_t | \leq 2K \vartheta( (T-t)/\etamin)$. 
Then 
$(F(t,H_t)-F(t,0))_{t\in[T-\delta,T)}\in L^\infty([T-\delta,T)\times \Omega;\mathbb R)$.
\end{Lemma}
\begin{proof}[Proof of the lemma]
Evoke that the generator $F$ is given by \eqref{eq:gene_F}. We have for all $t \in [0,T)$ that 
\begin{align*}
F(t,H_t) -F(t,0)& = \frac{-\zeta_t }{\eta_t \phi'(A_t)} \left[ f \left(\phi(A_t)-\phi'(A_t) \dfrac{H_t}{\zeta_t}\right) - f(\phi(A_t)) +f'( \phi(A_t)) \phi'(A_t) \dfrac{H_t}{\zeta_t}\right] \\
& =  \frac{\zeta_t}{\eta_t }  \dfrac{f(\phi(A_t) - \phi'(A_t)(\zeta_t)^{-1} H_t)}{-f(\phi(A_t))} +  \frac{\zeta_t}{\eta_t }  -  \frac{1}{\eta_t } \dfrac{\phi^{(2)}(A_t)}{\phi'(A_t)} H_t\\
& =  \frac{\zeta_t}{\eta_t }  \dfrac{f(\phi(A_t) - \phi'(A_t) (\zeta_t)^{-1}H_t)}{-f(\phi(A_t))} +  \frac{\zeta_t}{\eta_t } + \kappa^1_t \dfrac{H_t}{T-t}.
\end{align*}
From our estimate \eqref{eq:estim_vartheta} on $\vartheta$, we know that $t \mapsto \vartheta( (T-t)/\etamin) / (T-t)$ is bounded on $[T-\delta,T)$. 
Thus $\dfrac{H}{T-\cdot}$ is bounded on $[T-\delta,T) \times \Omega$. 

Now remind that $-f$ is non-decreasing on $\mathbb R$ and non-negative on $[0,\infty)$ (Condition~\ref{A3}). 
Since $-\phi' \geq 0$ and $\lvert H_t \rvert \le 2 K \vartheta((T-t)/\etamin) \le 2K \vartheta(\eta^\sharp A_t)$ for all $t \in [T-\delta,T]$, we therefore obtain for all $t \in [T-\delta,T)$ that 
 \begin{align*}
 \dfrac{-f(\phi(A_t) + 2K (\zeta_t)^{-1} \phi'(A_t) \vartheta(\eta^\sharp A_t))}{-f(\phi(A_t))} &  \leq \dfrac{ -f(\phi(A_t) - \phi'(A_t) (\zeta_t)^{-1}H_t) }{-f(\phi(A_t))} \\
 & \leq  \dfrac{-f(\phi(A_t)- 2K (\zeta_t)^{-1} \phi'(A_t)  \vartheta(\eta^\sharp A_t)) }{-f(\phi(A_t))}.
\end{align*}
Since $0\leq 2K  (\zeta_t)^{-1}  \leq 2K  (\zeta_\star)^{-1}   \leq \varsigma$ for all $t \in [0,T]$ 
and since by $\delta < \etamin \epsilon$ 
we have $0<  A_t < \epsilon$ for all $t \in [T-\delta,T)$, 
the conclusion thus follows from~\ref{A6} and Lemma~\ref{lem:cond_bound} in the appendix.  
\end{proof}

The previous lemma allows to define for all $\delta \in (0, \etamin \epsilon)$ by
\[
\Gamma(H)=\left(\mathbb E_t\bigg[\int_t^T F(s,H_s)\,ds\bigg] = \mathbb E_t\bigg[\int_t^T F(s,H_s)-F(s,0)\,ds\bigg] +\Gamma^0_t \right)_{t\in[T-\delta,T]}
\]
the operator $\Gamma:\overline B_{\mathcal H^{\delta,\theta}}(2K)\to L^\infty(\Omega;C([T-\delta,T];\mathbb R))$.

\medskip

\noindent {\bf Step 4.} 
In the following technical lemma, we establish some kind of local Lipschitz property of the generator. 
\begin{Lemma} \label{lemma-locally-Lip} 
Assume that Conditions {\rm \ref{A1}} to {\rm \ref{A6}} are satisfied. Let $\delta \in (0, \etamin \epsilon)$. Then 
we have for any $H$ and $\widehat H$ in $\overline B_{\mathcal H^{\delta,\theta}}(2K)$ 
that 
\[
|F(t,H_t)-F(t,\widehat H_t)|\leq \frac{1}{\eta_t }\Psi(A_t) |H_t-\widehat H_t| \qquad \forall t\in[T-\delta,T)\ ,\ a.s. 
\]
\end{Lemma}
\begin{proof}[Proof of the lemma]
Note first that for all $t\in [T-\delta,T)$ and $h\in\mathbb{R}$ with $|h| \leq  2K \vartheta((T-t)/\etamin)$ we have $\phi(A_t)-\phi'(A_t) (\zeta_t)^{-1}h > 0$ (cf.\ Lemma~\ref{lem:lastcondwelldef}). 
From \eqref{eq:gene_F} we have for all $t\in [T-\delta,T)$ and $|h| \leq  2K \vartheta((T-t)/\etamin)$ that 
$$\frac{\partial F}{\partial h}(t,h)= \frac{1}{\eta_t } \left[ f'(\phi(A_t)-\phi'(A_t) (\zeta_t)^{-1}h) - f'(\phi(A_t)) \right] .$$
Since $f$ is concave, $f'$ is non-increasing. Thereby for all $t\in [T-\delta,T)$ and $|h|\leq 2K\vartheta((T-t)/\etamin)$ we obtain
\begin{align*}
 \left|  \frac{\partial F}{ \partial h}(t,h) \right| &  \leq \frac{1}{\eta_t }  \max   \left\{ \left[ f'(\phi(A_t))- f'(\phi(A_t)-2K \phi'(A_t) (\zeta_t)^{-1}\vartheta(\eta^\sharp A_t))  \right] , \right. \\
&\qquad \qquad \qquad \left. \left[ f'(\phi(A_t) + 2K \phi'(A_t) (\zeta_t)^{-1}\vartheta(\eta^\sharp A_t)) - f'(\phi(A_t)) \right] \right\} \\
& \leq \frac{1}{\eta_t }  \max   \left\{ \left[ f'(\phi(A_t))- f'(\phi(A_t)-\varsigma \phi'(A_t) \vartheta(\eta^\sharp A_t))  \right] , \right. \\
&\qquad \qquad \qquad \left. \left[ f'(\phi(A_t) + \varsigma \phi'(A_t)\vartheta(\eta^\sharp A_t)) - f'(\phi(A_t)) \right] \right\} 
\end{align*}
Since $\delta < \etamin \epsilon$,  we obtain $0<  A_t < \epsilon$ for all $t \in [T-\delta,T)$, and hence Assumption \ref{A6} ensures for all $|h|\leq 2K\theta(T-t)$ and $t \in [T-\delta,T)$ that 
$$
 \left|  \frac{\partial F}{ \partial h}(t,h) \right|\le  \frac{1}{\eta_t }\Psi(A_t).
$$
The assertion then follows by the mean value theorem.
\end{proof}

\medskip

\noindent {\bf Step 5.} Now we can establish the existence of a unique solution $(H,Z^H)$ on the time interval $[T-\delta,T]$ for some $\delta < \etamin \epsilon$ by showing that the operator $\Gamma$ introduced in Step~3 is a contraction.  
For any $\delta \in (0, \etamin \epsilon)$ and $H,\widehat H\in\overline B_{\mathcal H^{\delta,\theta}}(2K)$, Step~4 yields for all $t\in[T-\delta,T]$ that 
\begin{align*}
|\Gamma(H)_t-\Gamma(\widehat H)_t|&\leq \mathbb E_t\bigg[\int_t^T| F(s,H_s)- F(s,\widehat H_s)|\,ds \bigg]\\
&\leq  \dfrac{1}{\etamin} \|H-\widehat H\|_{\mathcal H^{\delta,\theta}} \, \mathbb E_t\bigg[\int_t^T \Psi(A_s) \vartheta \bigg( \dfrac{T-s}{\etamin} \bigg)ds \bigg]\\
& \leq   \dfrac{1}{\etamin}  \|H-\widehat H\|_{\mathcal H^{\delta,\theta}} \, \vartheta \bigg( \dfrac{T-t}{\etamin} \bigg) \mathbb E_t\bigg[\int_t^T \Psi\bigg( \dfrac{T-s}{\etamax} \bigg) ds \bigg]\\
& \leq \eta^\sharp \|H-\widehat H\|_{\mathcal H^{\delta,\theta}} \,\vartheta \bigg( \dfrac{T-t}{\etamin} \bigg) \int_0^{\frac{T-t}{\etamax}} \Psi(u)du,
\end{align*}
where we have used 
the assumption that $\Psi$ is non-increasing. 
We can choose $\delta \in (0, \etamin \epsilon)$ (which we fix for the remainder of the proof) in such a way that for any $t \in [T-\delta,T]$ it holds  
$$ \eta^\sharp \int_0^{\frac{T-t}{\etamax}} \Psi(u)du \leq \dfrac{1}{2}.$$
Hence, $\Gamma$ is a $1/2$-contraction on $\overline B_{\mathcal H^{\delta,\theta}}(2K)$. 
Furthermore, for any $H\in\overline B_{\mathcal H^{\delta,\theta}}(2K)$, it follows with \eqref{eq:init_value} for all $t \in [T-\delta,T]$ that 
\begin{align*}
|\Gamma(H)_t|&\leq |\Gamma(H)_t-\Gamma^0_t|+|\Gamma^0_t|\\
	&\leq \frac{1}{2}  2K \vartheta \bigg( \dfrac{T-t}{\etamin} \bigg)  + K \vartheta \bigg( \dfrac{T-t}{\etamin} \bigg)=2K \vartheta \bigg( \dfrac{T-t}{\etamin} \bigg)  .
\end{align*}
We have proved that $\Gamma$ maps $\overline B_{\mathcal H^{\delta,\theta}}(2K)$ into itself and is a contraction. 
Using the properties of the map $\Gamma$, we deduce 
that there exists a unique process $H  \in \mathcal H^{\delta,\theta}$ such that a.s.\ for any $t \in [T-\delta,T]$ it holds 
$$H_t = \mathbb E_t\bigg[\int_t^T F(s,H_s)\,ds \bigg].$$
By the martingale representation theorem, we obtain the existence of $Z$ such that $(H,Z)$ solves~\eqref{eq:BSDE_F}.
Since $H \in \mathcal H^{\delta,\theta}$, from Lemma~\ref{lem:bounded_generator}, we deduce that the martingale $\int_{T-\delta}^\cdot Z dW$ is a BMO martingale (see~\cite{bech:06}). 

\medskip

\noindent {\bf Step 6.} Now $H_{T-\delta}$ is a bounded random variable.  If we consider the BSDE~\eqref{eq:BSDE_F} starting at the time $T-\delta$ from the terminal condition $H_{T-\delta}$, we can apply \cite[Proposition 5.24]{pard:rasc:14} to obtain a unique solution $(H,Z)$ on $[0,T-\delta]$ such that $H$ is bounded. The BMO property is implied by the boundedness of $H$.

\medskip

\noindent {\bf Step 7.} In this final step, we want to prove that the relation~\eqref{eq:asymp_beha_Y} holds.  Denote by $(H,Z^H)$ the solution of the BSDE~\eqref{eq:BSDE_H} constructed above and note that, by Step 1 and Step 5, for all $t\in[T-\delta, T]$ it holds 
$$\lvert H_t \rvert \le \frac{2K}{\zeta_\star} \vartheta \bigg( \dfrac{T-t}{\etamin} \bigg).$$ 
Let us define on $[0,T)$ the process $\widehat Y$ by 
$$\widehat Y_t =  \phi \left(\dfrac{T-t}{\eta_t} \right) - \phi'\left( \dfrac{T-t}{\eta_t} \right) H_t, \quad \forall t\in [0,T).$$ 
By construction, there exists a process $\widehat{Z}$ such that $(\widehat Y, \widehat{Z})$ satisfies the dynamics of the BSDE \eqref{eq:sing_BSDE} on $[0,T)$. However since $H$ may be negative, it is not clear that $\widehat Y$ has a limit at time $T$ and we cannot use a comparison principle (as in \cite[Proposition 20]{grae:popi:21}) to obtain that $\widehat Y = Y$. 

We first show that the a.s.\ limit at time $T$ of $\widehat Y$ is $\infty$. We have for all $t \in [T-\delta,T)$ that  
$$ - \dfrac{\phi'\left( \dfrac{T-t}{\eta_t} \right)}{\phi \left(\dfrac{T-t}{\eta_t} \right)} |H_t|  
\leq  \frac{2K}{\zeta_\star} \dfrac{-\phi'\left( \dfrac{T-t}{\eta_t} \right)}{\phi \left(\dfrac{T-t}{\eta_t} \right)} \vartheta \left( \dfrac{T-t}{\etamin} \right) 
\leq \frac{2K}{\zeta_\star} \dfrac{-\phi'\left( \dfrac{T-t}{\etamax} \right)}{\phi \left(\dfrac{T-t}{\etamin} \right)} \vartheta \left( \dfrac{T-t}{\etamin} \right).$$
From \cite[Lemma 11 and Equation (38)]{grae:popi:21}, we know that there exists $C>0$ and $\delta' \in (0,\delta)$ (depending on Condition \ref{A5}) such that for all $t \in [T-\delta',T)$ it holds 
$$-\phi'\left( \dfrac{T-t}{\etamax} \right) \leq - (\eta^\sharp)^C  \phi'\left( \dfrac{T-t}{\etamin} \right) .$$
Let $\Upsilon$ be the function $\dfrac{-\phi'}{\phi} \vartheta$.
Then we have for all $t \in [T-\delta',T)$ that 
$$ - \dfrac{\phi'\left( \dfrac{T-t}{\eta_t} \right)}{\phi \left(\dfrac{T-t}{\eta_t} \right)} |H_t|  
\leq  \frac{2K}{\zeta_\star} (\eta^\sharp)^C \Upsilon \left( \dfrac{T-t}{\etamin} \right) .$$
If $x\ge 0$ satisfies $\vartheta(x) = \varpi(x)$, evoke $0\leq \varpi(x) \leq  \dfrac{ x }{-\phi'\left( x\right) }$ and thus $ \Upsilon (x) \leq  \dfrac{x}{\phi(x)}$. 
For the other case (i.e., $\vartheta(x) = x^2$), note that \ref{A5} and Lemma~\ref{lem:ineg_kappa_i_x} yield the existence of some $c>0$ such that for all $x \in (0,\epsilon)$ we have 
$\Upsilon (x) \leq  c x$. In both cases $\Upsilon(x)$ tends to zero when $x$ goes to zero. Therefore 
$$\lim_{t\to +\infty} \widehat Y_t = + \infty.$$

Then arguing as in the proof of minimality of $Y$ (see \cite[Proposition 4]{krus:popi:15}) leads to $Y \leq \widehat Y$: we use a comparison principle between $(Y^n)_{n\in\mathbb{N}}$, i.e.\ the approximating sequence of $Y$, and~$\widehat Y$. 

To obtain the relation \eqref{eq:asymp_beha_Y}, we need to prove the converse inequality. 
We introduce a process $\mathcal H$ such that 
\begin{equation}\label{eq:definitionofmathcalHviaY}
    Y_t =  \phi \bigg(\dfrac{T-t}{\eta_t} \bigg) - \phi'\bigg( \dfrac{T-t}{\eta_t} \bigg) \mathcal H_t, \quad \forall t \in [0,T).
\end{equation}
Note that since $Y\le \widehat Y$, we have for all $t \in [0,T)$ that $\mathcal{H}_t \le H_t$. 
To derive a lower bound for $\mathcal H$, 
%
%
evoke that from \cite[Lemma 6]{grae:popi:21} we have for all $t \in [0,T)$ that  
$$Y_t \geq \phi \bigg( \mathbb E_t\bigg[ \int_t^T \dfrac{1}{\eta_s} ds \bigg] \bigg).$$
Moreover, it holds for all $t \in [0,T)$ that 
$$\mathbb E_t\bigg[ \int_t^T \dfrac{1}{\eta_s} ds \bigg] = \dfrac{T-t}{\eta_t} - \int_t^T (T-u) \mathbb E_t[ \varrho_u ] du,$$
with 
$$\varrho_u := \dfrac{1}{\eta_u} \left( \dfrac{b^\eta_u}{\eta_u} -  \left(  \dfrac{\sigma^\eta_u}{\eta_u}\right)^2 \right) \geq - \dfrac{1}{\etamin}  \left\| \dfrac{b^\eta_u}{\eta_u} - \left(  \dfrac{\sigma^\eta_u}{\eta_u}\right)^2 \right\|_\infty  \geq - \dfrac{2K}{(\etamin)^2} .$$
Hence we have for all $t\in[0,T)$ that 
\begin{align*}
Y_t & \geq \phi\bigg(  \dfrac{T-t}{\eta_t} - \int_t^T (T-u) \mathbb E_t[ \varrho_u ] du \bigg)  
\geq  \phi \bigg(  \dfrac{T-t}{\eta_t} + \dfrac{K}{(\etamin)^2} (T-t)^2  \bigg)\\
& = \phi \bigg(  \dfrac{T-t}{\eta_t} \bigg) + \dfrac{K}{(\etamin)^2} (T-t)^2  \int_0^1 \phi' \bigg(  \dfrac{T-t}{\eta_t}+a   \dfrac{K}{(\etamin)^2}  (T-t)^2  \bigg) da \\
& \geq  \phi \bigg(  \dfrac{T-t}{\eta_t} \bigg) + \dfrac{K}{(\etamin)^2} (T-t)^2  \phi' \bigg(  \dfrac{T-t}{\eta_t} \bigg). 
\end{align*}
It follows for all $t\in [0,T)$ that 
$$\mathcal H_t \geq - \dfrac{K}{(\etamin)^2}  (T-t)^2 .$$
In particular, for all $t\in[0,T)$, the negative part of $\mathcal H_t$ is smaller than or equal to 
$$\dfrac{K}{(\etamin)^2}  (T-t)^2 \leq K \vartheta \bigg( \dfrac{T-t}{\etamin} \bigg) 
\le \frac{2K}{\zeta_\star} \vartheta \bigg( \dfrac{T-t}{\etamin} \bigg) .$$ 
In addition, we can combine this with $\mathcal{H} \le H$ to obtain a.s.\ $\lim_{t \to T} \mathcal H_t = 0$. 
Moreover, note that from \eqref{eq:definitionofmathcalHviaY} we now deduce that there exists a process $\mathcal Z$ such that $(\mathcal H,\mathcal Z)$ is a solution of the BSDE~\eqref{eq:BSDE_H} with the terminal condition~$0$. 
Using \eqref{eq:BSDE_H}, for any $t \in [0,T]$, we have 
\begin{align*}
d(H_t- \mathcal H_t) & = \frac{1}{\eta_t \phi'\left( A_t \right) } \left[ f( \phi(A_t) - \phi'(A_t) H_t) -f( \phi(A_t) - \phi'(A_t) \mathcal H_t) \right] dt   \\ 
&  \quad + \frac{ f'(\phi(A_t))}{\eta_t }  (H_t- \mathcal H_t ) dt  - \mu_t (H_t- \mathcal H_t ) dt + (Z^H_t - \mathcal Z_t) \left( dW_t - \kappa^1_t  \dfrac{\sigma^\eta_t }{\eta_t}  dt \right) \\
& = \frac{1}{\eta_t \phi'\left( A_t \right) } \bigg[ f( \phi(A_t) - \phi'(A_t) H_t) -f( \phi(A_t) - \phi'(A_t) \mathcal H_t)  \\
& \hspace{3cm} +  f'( \phi(A_t) - \phi'(A_t) \mathcal H_t) \phi'(A_t) (H_t- \mathcal H_t )\bigg] dt   \\ 
& \quad + \frac{ f'(\phi(A_t)) -  f'( \phi(A_t) - \phi'(A_t) \mathcal H_t)}{\eta_t }  (H_t- \mathcal H_t ) dt  \\
& \quad  - \mu_t (H_t- \mathcal H_t ) dt + (Z^H_t - \mathcal Z_t) \left( dW_t - \kappa^1_t  \dfrac{\sigma^\eta_t }{\eta_t}  dt \right)  \\
& \leq \frac{ f'(\phi(A_t)) -  f'( \phi(A_t) - \phi'(A_t) \mathcal H_t)}{\eta_t }  (H_t- \mathcal H_t ) dt  \\
& \quad - \mu_t (H_t- \mathcal H_t ) dt + (Z^H_t - \mathcal Z_t) \left( dW_t - \kappa^1_t  \dfrac{\sigma^\eta_t }{\eta_t}  dt \right),
\end{align*}
where the last inequality is due to the concavity of the function $f$. From our estimates on $\mathcal H$ and Assumption \ref{A6}, we obtain (similar to the proofs of Lemma~\ref{lem:bounded_generator} and Lemma~\ref{lemma-locally-Lip}) for all $t\in[T- \delta,T)$ that 
\begin{equation}\label{eq:BoundUsingPsiFromA6}
    \left| \frac{ f'(\phi(A_t)) -  f'( \phi(A_t) - \phi'(A_t) \mathcal H_t)}{\eta_t } \right| \leq \dfrac{1}{\etamin} \Psi( A_t)\leq \dfrac{1}{\etamin} \Psi\bigg( \dfrac{T-t}{\etamax}  \bigg) .
\end{equation}
Moreover, using in addition that $H$ is bounded, we see that for all $t \in [0,T- \delta]$ the left-hand side in \eqref{eq:BoundUsingPsiFromA6} is bounded. 
We thus obtain that there exists an integrable function $\widetilde\Psi$ on $[0,\infty)$ such that for all $t \in [0,T]$ we have 
$$d(H_t- \mathcal H_t  )\leq \bigg[ \dfrac{1}{\etamin} \widetilde\Psi\bigg( \dfrac{T-t}{\etamax}  \bigg) -\mu_t \bigg] (H_t- \mathcal H_t ) dt + (Z^H_t - \mathcal Z_t) \left( dW_t - \kappa^1_t  \dfrac{\sigma^\eta_t }{\eta_t}  dt \right).$$
Using the explicit formula for a linear BSDE, we deduce that $H \leq \mathcal H$. Hence $\widehat Y \leq Y$ and the desired result is proved. As a by-product, we also obtain that $(H,Z^H)$ is the unique solution of the BSDE~\eqref{eq:BSDE_H}. 
This completes the proof of Theorem \ref{thm:gene_case}.

\subsection{Examples} \label{ssect:examples}

We present two examples where Assumptions \ref{A3} to \ref{A6} hold:
\begin{itemize}
\item[(i)] Power case: $f(y) = - y |y|^{q-1} $ for some $q>1$.
\item[(ii)] Exponential case:  $f(y) = -(\exp(ay)-1)$ for some $a> 0$.
\end{itemize}

\subsubsection{Power case} \label{ssect:power_case}

This case plays a key role in the liquidation problem. The function satisfies {\rm \ref{A3}}. Moreover with $p$ the H\"older conjugate of $q$,  we can explicitly compute all required quantities:
\begin{align*}
f'(y) & = - q |y|^{q-1}, \quad f^{(2)} (y ) = -q(q-1) |y|^{q-2} \mbox{sgn}(y),\\
G(x) & = (p-1)x^{1-q} ,\quad \phi(x)  = [(q-1)x]^{1-p} ,\quad  \phi'(x)= - [(q-1)x]^{-p}, \\
\kappa^0(x) & = \dfrac{-\phi'(x) x}{\phi(x)} = p-1,\quad \kappa^1(x) =p,\quad \kappa^2(x) = p+1 .
\end{align*}
Thus {\rm \ref{A4}} and {\rm \ref{A5}} hold. 
Furthermore, for $x$ close to $0$, we have 
\begin{align*}
    \varpi(x) & = \dfrac{(q-1)^p}{p+1} x^{p+1}, \quad \vartheta(x) = x^2 .
\end{align*}
Observe that for any $\varsigma>0$ there exists $R>0$ such that on $(R,+\infty)$ the function 
$$y \mapsto  \dfrac{f^{(2)}[y \pm \varsigma f(y) (G(y))^{2}]}{f^{(2)}(y)}
= \big(1\mp \varsigma (p-1)^2 y^{-q}\big)^{q-2} $$
is bounded. Hence, Lemma~\ref{lem:checking_A6} yields that {\rm \ref{A6}} holds with $\Psi$ a constant function. 


Expansion \eqref{eq:asymp_beha_Y} becomes
$$Y_t =\left( (p-1)\dfrac{\eta_t}{T-t} \right)^{p-1} +(p-1)^p \dfrac{\eta_t^p}{(T-t)^p}  H_t,\quad \forall t \in [0,T), $$
where $H_t = O((T-t)^2)$. 
We deduce that for $p < 2$, that is $q > 2$, 
$$Y_t = (p-1)^{p-1} \dfrac{\eta_t^{p-1}}{(T-t)^{p-1}} + O((T-t)^{2-p})$$
and Condition \eqref{eq:cond_term} holds with $\alpha = 2 - p >0$. However for $q \leq 2$, we cannot control $(T-t)^{-p} H_t$: 
$$ \lim_{t \to T}  -\phi'(T-t) (T-t)^2 \neq 0.$$

\subsubsection{Exponential case}

Here $f(y) = -(\exp(ay)-1)$ for some $a> 0$. Then \ref{A3} and \ref{A4} are true with 
$$G(y) =-\dfrac{1}{a} \log (1- \exp(-ay)),\quad \phi(x) = - \frac{1}{a} \log \left(1-e^{-ax}\right),\quad \phi'(x) = -  \dfrac{e^{-ax}}{1-e^{-ax}} .$$
It is proved in \cite[Section~3.1]{grae:popi:21} that $\kappa^1(x)$ (resp.\ $\kappa^2(x)$) is equivalent to $1$ (resp.\ $2$) when $x$ goes to 0. Thus \ref{A5} also holds. Notice that  
$$\kappa^0(x) = -\dfrac{\phi'(x)}{\phi(x) } x=  - a \dfrac{xe^{-ax}}{1-e^{-ax}}  \big(\log(1-e^{-ax})\big)^{-1} \underset{x\to 0}{\sim} - \big(\log(1-e^{-ax})\big)^{-1}.$$
Moreover we have for all $x\ge 0$ that 
$$\varpi(x) = \frac{1}{a} (e^{ax}-1 - ax).$$
On some neighborhood of zero, $\varpi$ is equivalent to $x\mapsto (a/2)x^2$. Or in other words, $-\phi'(x) x^2\sim 2x/a $. Somehow, the exponential case is a kind of critical case, where $\varpi(x)$ and $x^2 $ are equivalent.  The rate of convergence $\vartheta(x)$ is $x^2$ as in the power case.

Here $f(y) = -(\exp(ay)-1)$, $\vartheta(x) = x^2$ and $f^{(2)}(y) = -a^2 \exp(ay)$. 
We thus obtain for all $\varsigma>0$ and $y>0$ that 
$$ \dfrac{f^{(2)}[y \pm \varsigma f(y) (G(y))^{2}]}{f^{(2)}(y)} = \exp \left( \pm a \varsigma f(y) (G(y))^{2} \right),$$ 
and 
$$-f(y) (G(y))^2 =\dfrac{1}{a^2}  (\exp(ay)-1) (\log (1- \exp(-ay)))^2 \underset{y \to \infty}{\sim} \dfrac{1}{a^2}  (e^{ay}-1) e^{-2ay} $$ 
is bounded (see also Lemma \ref{eq:lemma_exp_decreasing_fct}). Therefore using Lemma \ref{lem:checking_A6}, we conclude that \ref{A6} holds. 

Since $\vartheta$ is equivalent to $\varpi$ close to $0$, we deduce from~\eqref{eq:bound_second_term} that Condition \eqref{eq:cond_term} holds with $\alpha = 1$.

\subsection{Comparing the generator with the exponential case} \label{ssect:vartheta}

We already mention above that the exponential case is a kind of critical case. Indeed if we can ``compare'' $f$ with an exponential, then we can deduce whether $\vartheta(x) = \varpi(x)$ or $\vartheta(x) =x^2$. 

\begin{Lemma} \label{eq:lemma_exp_decreasing_fct}
Assume Condition~{\rm \ref{A3}}, that $f(x)<0$ for all $x>0$, and that there exist $R>0$ and $\beta>0$ such that 
$y\mapsto f(y)\exp(-\beta y)$ is non-increasing on $(R,+\infty)$.
%
Then {\rm \ref{A4}} holds and 
$x\mapsto -\phi'(x) x $ is bounded by $1/\beta$ on $(0,G(R))$. In particular, there exists $C>0$ such that $0\leq \phi(x) \leq C -\dfrac{1}{\beta} \log(x)$ for all $x \in (0,G(R))$. 
\end{Lemma}
\begin{proof}
For any $R < y \le w$, it holds that $f(y)e^{-\beta y} \geq f(w) e^{- \beta w}$ and thus 
$$\dfrac{f(y)}{f(w)} \leq e^{-\beta w+\beta y} .$$
First we obtain that $[x,+\infty) \ni w \mapsto -1/f(w)$ is integrable for any $x>0$, that is \ref{A4} holds. Then using the above relation in the integration leads to
$-f(y) G(y) \leq \dfrac{1}{\beta}$ for all $y>R$.
Since $G$ is non-increasing with inverse $\phi$ and $-\phi'(x) x =-f(\phi(x)) x$ for all $x>0$, a change of variables $y=\phi(x)$ yields that 
$x\mapsto -\phi'(x) x $ is bounded by $1/\beta$ on $(0,G(R))$. 
Moreover, integrating with respect to $x$ on both sides of $-\phi'(x) \le \frac{1}{x} \frac{1}{\beta}$ yields the bound for~$\phi$. 
\end{proof}

From this lemma we deduce that 
$$\dfrac{x^2}{2} = \int_0^x z dz \leq  \dfrac{1}{\beta} \varpi(x)$$ 
for sufficiently small $x>0$.
Therefore we deduce from \eqref{eq:bound_second_term} that, for $t$ close to $T$, we in this case have  
$$Y_t =\phi \left( \dfrac{T-t}{\eta_t} \right) + O(T-t).$$

We know that there are examples (the power or exponential cases) where $\varpi$ is dominated by $x \mapsto x^2$ close to $0$. In these cases, the estimate in \eqref{eq:bound_second_term} is of 
order $O(\phi'(T-t) (T-t)^2)$ for $t$ close to $T$. 
As an analogue to Lemma~\ref{eq:lemma_exp_decreasing_fct}, we have: 
\begin{Lemma} \label{eq:lemma_exp_increasing_fct}
Assume that Conditions~{\rm \ref{A3}} and~{\rm \ref{A4}} are satisfied and that there exist $R>0$ and $\beta>0$ such that 
$y\mapsto f(y)\exp(-\beta y)$ is non-decreasing on $(R,+\infty)$.
Then 
$x\mapsto -\phi'(x) x $ is bounded away from zero on $(0,G(R))$ and for all $x\in(0,G(R))$ it holds $ \varpi(x) \leq (\beta /2)x^2$. 
\end{Lemma}
\begin{proof}
    The claim can be proven using similar arguments as in the proof of Lemma~\ref{eq:lemma_exp_decreasing_fct}.
\end{proof}

\section{Conclusion and perspectives} \label{sect:power_case_expansion}

In this paper, we provide a complete algorithm, based on Assumption~\ref{ass:approximation}, to compute the solution of a BSDE with terminal condition equal to $+\infty$ for a large class of generators $f$. To check Assumption~\ref{ass:approximation}, we prove an asymptotic expansion \eqref{eq:asymp_beha_Y} for the solution $Y$ near the terminal time $T$. 

Hence we are able to obtain the value function together with the optimal control for the related liquidation problem to minimize~\eqref{eq:liquidationproblemcosts} when $p < 2$. Notice that it covers the case of~\cite{almg:thum:haup:05}, where the authors analyze a large data set from the Citigroup US equity trading desks and show that $p=1.6$ is a good estimate of this parameter. 

When $p \geq 2$, a further expansion is required. Let us briefly explain how it can be done. Evoke that in \eqref{eq:asymp_beha_Y} we have in the power case that 
\begin{equation*} 
Y_t =  \left(\dfrac{\eta_t}{(q-1)(T-t)} \right)^{p-1} + \left(\dfrac{1}{(q-1)(T-t)} \right)^{p}  \widehat H_t,\quad \forall t \in [0,T),
\end{equation*}
with $\widehat H = \eta^p H$. The dynamics of $\widehat H$ follow from the BSDE \eqref{eq:BSDE_H} for $H$ and Assumption~\ref{A2} for $\eta$ and with similar arguments as in \cite[Section 4.1]{grae:popi:21}, we can prove that 
\begin{align*}
Y_t  & =  (p-1)^{p-1}  \dfrac{\eta_t^{p-1}}{(T-t)^{p-1}} + \dfrac{(p-1)^p}{(T-t)^{p}} \mathbb E_t\bigg[  \int_t^T (T-s) \eta_s^{p-1} \bigg(  \dfrac{b^\eta_s}{\eta_s}  + \left(   \dfrac{p}{2} -1 \right) \left( \dfrac{\sigma^\eta_s}{\eta_s} \right)^2  \bigg) ds  \bigg]  \\
& \quad +  (T-t)^{-p} \mathbb E_t\bigg[  \int_t^T  (T-s)^p \lambda_s ds \bigg] + O((T-t)^{3-p}). 
\end{align*}
Since for all $t \in [0,T)$ we have 
$$ (T-t)^{-p} \mathbb E_t\bigg[  \int_t^T  (T-s)^p |\lambda_s| ds  \bigg]  \leq \dfrac{\lambda^\star }{p+1} (T-t),$$
we deduce for $2 \leq p < 3$ and any $t \in [0,T)$ that 
$$Y_t =  \xi_t +  O((T-t)^{3-p})$$
with 
\begin{equation}\label{eq:xiconclusion}
    \xi_t = (p-1)^{p-1}  \dfrac{\eta_t^{p-1}}{(T-t)^{p-1}} + \dfrac{(p-1)^p}{(T-t)^{p}} \mathbb E_t\bigg[  \int_t^T (T-s) \eta_s^{p-1} \bigg(  \dfrac{b^\eta_s}{\eta_s}  + \left(   \dfrac{p}{2} -1 \right) \left( \dfrac{\sigma^\eta_s}{\eta_s} \right)^2  \bigg) ds   \bigg] .
\end{equation}
Hence Assumption~\ref{ass:approximation} holds with $\alpha = 3-p > 0$. However, for a numerical scheme, the conditional expectation in~\eqref{eq:xiconclusion} 
has to be precisely computed, since for $2 < p< 3$ (that is $3/2 < q < 2$) this second term is of order $(T-t)^{2-p}$ and explodes as $t$ tends to $T$ (for $q=p=2$, with a continuity condition on $b^\eta$ and $\sigma^\eta$, this term has a finite equivalent as $t$ tends to $T$). 

When $p \geq 3$, we can iterate this procedure. But it requires to slot together more and more conditional expectations, which creates new issues for the numerical part. These questions are left for further research.

\appendix 
\section{Proofs of some technical results}
\label{sect:appendix}

Here some technical results are proved. 

We first show that the first derivative of $\phi$ is negative and the second derivative of $\phi$ is positive. 
\begin{Lemma}\label{lem:ineq_phi}
    Assume that Conditions {\rm \ref{A3}} and {\rm \ref{A4}} are satisfied. Then it holds for all $x>0$ that $\phi'(x)<0$ and $\phi''(x)>0$. 
\end{Lemma}
\begin{proof}
    Observe that $\phi'=f\circ \phi$ and \ref{A3} ensure for all $x >0$ that $\phi'(x)\ge 0$. 
    Suppose that there exists $x_0 >0$ such that $\phi'(x_0)=0$. This implies that $f(y_0)=0$ for $y_0=\phi(x_0)>0$ and thus, by \ref{A3}, that $f(x)=0$ for all $x \in [0,y_0]$. But this contradicts \ref{A4}, hence $\phi'(x)<0$ for all $x>0$.

    Note that for all $x>0$ it holds $\phi''(x)=f'(\phi(x)) \phi'(x)$. 
    Since $f$ is concave on $[0,+\infty)$ and $f$ is non-increasing, we have $f'(y) \leq f'(0) \leq 0$ for all $y \geq 0$. If there exists $y_0 > 0$ such that $f'(y_0)=0$, then the function $f$ would be constant equal to zero on $[0,y_0]$, which again contradicts \ref{A4}. 
    This proves for all $x >0$ that $f'(\phi(x))<0$ and completes the proof.
\end{proof}

\medskip

We use the previous result in the following proofs of Lemma~\ref{lem:ineg_kappa_i_x} and Lemma~\ref{lem:ineg_kappa_i}.

\noindent \begin{proof}[Proof of Lemma~\ref{lem:ineg_kappa_i_x}]
Evoke that for all $x > 0$ and $i\in\{0,1,2\}$ it holds  
$$\kappa^i(x) = - \dfrac{\phi^{(i+1)} \left(x \right)  }{\phi^{(i)}\left(x \right)}x$$
and recall that $\phi' = f \circ \phi$.
Lemma~\ref{lem:ineq_phi} immediately yields for all $x>0$ that $0\leq \kappa^0(x)$. 
Next, note for all $x>0$ that  
$$\kappa^1(x) - \kappa^0(x) = \dfrac{x\phi(x)}{-\phi'(x)} \left( \dfrac{\phi'(x)}{\phi(x)} \right)'.$$
This quantity is non-negative if $\phi' / \phi$ is non-decreasing, which is equivalent to $(0,\infty) \ni y \mapsto f(y) / y$ being non-increasing (variable change $y = \phi(x)$). Since $f$ is concave on $[0,\infty)$, we have for any $v > 0$ and $y > 0$ that  
$$\dfrac{f(v) - f(y)}{v-y} \leq \dfrac{f(v)-f(0)}{v-0} = \dfrac{f(v)}{v}.$$
Letting $v$ go to $y$, we deduce that for all $y > 0$ it holds $yf'(y) \leq f(y)$, that is $(0,\infty) \ni y \mapsto f(y) / y$ is non-increasing. 
Similarly, we consider for all $x>0$ the difference 
$$\kappa^2(x) - \kappa^1(x) =- \dfrac{\phi^{(3)} \left(x \right)  }{\phi^{(2)}\left(x \right)}x + \dfrac{\phi^{(2)} \left(x \right)  }{\phi'\left(x \right)}x = x \dfrac{\phi^{(3)} \left(x \right) \phi'(x) -  \left(\phi^{(2)}(x) \right)^2  }{-\phi'\left(x \right) \phi^{(2)}\left(x \right)}, $$
where 
$-\phi'(x)\phi^{(2)}(x)>0$ due to Lemma~\ref{lem:ineq_phi}. 
Observe moreover that   
$$\phi^{(3)}(x) = (f''\circ \phi)(x)(\phi'(x))^2 + ((f'\circ \phi)(x))^2 \phi'(x) $$
and 
\begin{align*} 
\phi^{(3)}(x) \phi'(x) & = (f''\circ \phi)(x)(\phi'(x))^3 + ((f'\circ \phi)(x))^2 ( \phi'(x))^2 \\
& \geq ((f'\circ \phi)(x))^2 ( \phi'(x))^2 = \big(\phi^{(2)}(x) \big)^2,
\end{align*}
since $f'' \leq 0$ because $f$ is concave on $[0,\infty)$. 
We have thus proved \eqref{ineq:kappa_i}.
\end{proof}

\noindent \begin{proof}[Proof of Lemma~\ref{lem:ineg_kappa_i}]
From Lemma~\ref{lem:ineg_kappa_i_x}, 
if the process $\kappa^2$ is bounded on $[0,T)\times \Omega$, the same holds for $\kappa^1$ and $\kappa^0$ and, with \ref{A1} and \ref{A2}, for $\mu$. 
Assumption \ref{A5}, together with \ref{A1}, implies that the process $\kappa^2$ is bounded on $[T-\etamin \frac{\epsilon}{2},T) \times \Omega$. 
On $[0,T-\etamin \frac{\epsilon}{2}] \times \Omega$ the process $A$ is bounded from below by $\frac{\etamin}{\etamax} \frac{\varepsilon}{2}$ and from above by $\frac{T}{\etamin}$. 
Since $f$ is of class $C^2$ on $[0,\infty)$, we have that $\phi$ is of class $C^3$. Therefore, $\phi^{(3)}$ is bounded on $[\frac{\etamin}{\etamax} \frac{\varepsilon}{2},\frac{T}{\etamin}]$. 
Moreover, continuity of $\phi^{(2)}$ and Lemma~\ref{lem:ineq_phi} ensure that $\phi^{(2)}$ is bounded on $[\frac{\etamin}{\etamax} \frac{\varepsilon}{2},\frac{T}{\etamin}]$ with a lower bound strictly greater than zero. 
We thus conclude that the process $\kappa^2$ is bounded also on $[0,T-\etamin \frac{\epsilon}{2}] \times \Omega$.
%
\end{proof}

Next we prove Lemma~\ref{lem:lastcondwelldef}.

\noindent \begin{proof}[Proof of Lemma~\ref{lem:lastcondwelldef}]
From \cite[Lemma 11 and Equation (38)]{grae:popi:21}, we know that there exists a constant~$k>0$ (note that $k$ is the bound on $x \mapsto \kappa_1(x)$ and depends on \ref{A5}) such that for all $x \in (0,(\eta^\sharp)^{-1}\epsilon)$ it holds 
\begin{equation*}
    1 \leq \dfrac{ -\phi'(x) }{-\phi'\left(\eta^\sharp x\right) }\leq (\eta^\sharp)^k. 
\end{equation*}
From \eqref{eq:estim_vartheta} and since $-\phi'$ is non-negative and non-increasing we thus obtain for all $x \in (0,(\eta^\sharp)^{-1}\epsilon)$ that 
$$  
0\leq-\phi'(x) \vartheta(\eta^\sharp x) \leq \eta^\sharp x  \max \left( \dfrac{ -\phi'(x) }{-\phi'\left(\eta^\sharp x\right) } , -\eta^\sharp x \phi'(x)\right) \leq \eta^\sharp x  \max \left( (\eta^\sharp)^k, \eta^\sharp \kappa^0(x) \phi(x)\right).
$$
Due to \ref{A5} and Lemma~\ref{lem:ineg_kappa_i_x}, we also know that $x \mapsto \kappa^0(x)$ is bounded on $(0,\epsilon)$. Hence 
there exists $\epsilon'>0$ such that for all $x \in (0,\epsilon')$ we have  
$$0\leq \varsigma(\eta^\sharp)^2 x \kappa^0(x) \leq \dfrac{1}{2}.$$
Notice that $\phi(x) + \varsigma \phi'(x) \vartheta(\eta^\sharp x) = \phi(x) - \varsigma (-\phi'(x)) \vartheta(\eta^\sharp x) $ for all $x >0$. 
Hence we deduce the desired result since $\phi$ tends to $\infty$ when $x$ goes to zero. 
\end{proof}

\medskip

Now let us prove Lemma~\ref{lem:checking_A6}. 

\noindent \begin{proof}[Proof of Lemma \ref{lem:checking_A6}]
For $x>0$ close to $0$ we have 
\begin{align*}
& f'(\phi(x))- f'(\phi(x) \pm \varsigma \phi'(x) \vartheta(x)) \\
& =\mp \varsigma \phi'(x) \vartheta(x)  \int_0^1 f^{(2)} (\phi(x) \pm \alpha \varsigma \phi'(x) \vartheta(x)) d\alpha \\
& =\mp \varsigma \phi'(x) f^{(2)} (\phi(x))  \vartheta(x)  \int_0^1 \dfrac{ f^{(2)} (\phi(x) \pm \alpha \varsigma \phi'(x) \vartheta(x) )}{f^{(2)} (\phi(x))}d\alpha \\
& =\mp \varsigma \left(    \kappa^2(x) \kappa^1(x) -(\kappa^1(x))^2\right)\int_0^1 \dfrac{ f^{(2)} (\phi(x) 
\pm \alpha  \varsigma \phi'(x) x^2) }{f^{(2)} (\phi(x))}d\alpha .
\end{align*}
Since $x\mapsto \kappa^1(x)$ and $x\mapsto\kappa^2(x)$ are bounded functions on $(0,\epsilon)$ (due to \ref{A5} and Lemma~\ref{lem:ineg_kappa_i_x}), using the change of variable $y=\phi(x)$, we obtain the wanted result. 
\end{proof}

Let us finish with the next result. 
\begin{Lemma}\label{lem:cond_bound}
Assume that Conditions {\rm \ref{A1}} to {\rm \ref{A6}} are satisfied. Then 
$$x \mapsto  \dfrac{ f(\phi(x) \pm \varsigma \phi'(x) \vartheta(\eta^\sharp x))}{ f(\phi(x))}$$ is bounded on $(0,\epsilon)$. 
\end{Lemma}
\begin{proof}
For any $x > 0$, since $f$ is non-increasing and negative on $(0,\infty)$, it holds 
$$ \dfrac{ f(\phi(x) - \varsigma \phi'(x) \vartheta(\eta^\sharp x))}{ f(\phi(x))}  \geq 1.$$ 
Moreover, for all $x \in(0,\epsilon)$ we have 
\begin{align*}
 & \dfrac{ f(\phi(x) - \varsigma \phi'(x) \vartheta(\eta^\sharp x))}{ f(\phi(x))} - 1 
  =- \dfrac{ \varsigma \phi'(x) \vartheta(\eta^\sharp x)}{ f(\phi(x))} \int_0^1 f'( \phi(x) - a \varsigma \phi'(x) \vartheta(\eta^\sharp x)) d a \\
& \quad =-  \varsigma \vartheta(\eta^\sharp x) \int_0^1 f'( \phi(x)- a \varsigma \phi'(x) \vartheta(\eta^\sharp x)) d a \\
& \quad = -  \varsigma \vartheta(\eta^\sharp x) \int_0^1\left[  f'( \phi(x)- a \varsigma \phi'(x) \vartheta(\eta^\sharp x)) -f'( \phi(x))\right]  d a - \varsigma  f'( \phi(x))\vartheta(\eta^\sharp x)\\
&  \quad = -  \varsigma \vartheta(\eta^\sharp x) \int_0^1 \left[ f'( \phi(x)- a \varsigma \phi'(x) \vartheta(\eta^\sharp x)) -f'( \phi(x))\right]  d a + \varsigma  \kappa^1(x) \dfrac{\vartheta(\eta^\sharp x)}{x}\\
& \quad \leq \varsigma \vartheta(\eta^\sharp x)\Psi(x) + \varsigma  \kappa^1(x) \dfrac{\vartheta(\eta^\sharp x)}{x}.
\end{align*}
Similarly, it holds for all $x \in(0,\epsilon)$ that 
 \begin{align*}
 &0 \geq  \dfrac{ f(\phi(x) + \varsigma \phi'(x) \vartheta(\eta^\sharp x))}{ f(\phi(x))} - 1 
  = \dfrac{ \varsigma \phi'(x) \vartheta(\eta^\sharp x)}{ f(\phi(x))} \int_0^1 f'( \phi(x) + a \varsigma \phi'(x) \vartheta(\eta^\sharp x)) d a \\
&  \quad =   \varsigma \vartheta(\eta^\sharp x) \int_0^1 \left[ f'( \phi(x) + a \varsigma \phi'(x) \vartheta(\eta^\sharp x)) -f'( \phi(x))\right]  d a - \varsigma  \kappa^1(x) \dfrac{\vartheta(\eta^\sharp x)}{x}\\
& \quad \geq 
- \varsigma  \kappa^1(x) \dfrac{\vartheta(\eta^\sharp x)}{x}.
\end{align*}
The function $x\mapsto \kappa^1(x)$ is bounded on $(0,\epsilon)$ due to Condition~\ref{A5} (cf.~\eqref{ineq:kappa_i}), and $x \mapsto\dfrac{\vartheta(\eta^\sharp x)}{x}$ is also bounded on $(0,\epsilon)$. The conclusion of the lemma now directly follows from Assumption~\ref{A6}.
\end{proof}

\paragraph{Acknowledgments}

All authors acknowledge funding by the Deutsche Akademische Austauschdienst (DAAD, German Academic Exchange Service) within the funding program ``Programme for Pro\-ject-Related Personnel Exchange with France'' -- Project number 57702378, and by Campus France PHC (Projet Hubert Curien) Procope number 50835ZC. 
Julia Ackermann and Thomas Kruse acknowledge funding by the Deutsche For\-schungsgemeinschaft (DFG, German Research Foundation) -- Project-ID 531152215 -- CRC 1701.


\bibliography{biblio_num}

\end{document}